\documentclass[11pt]{amsart}

\usepackage{amsmath, amssymb, enumerate, amsthm, bbm, dsfont, setspace, tikz-cd, hyperref, amsfonts, amscd, amssymb, latexsym, mathrsfs, comment}
\usepackage[shortcuts]{extdash}
\usepackage[all]{xy}

\theoremstyle{plain}
\newtheorem{theorem}{Theorem}[section]

\newtheorem*{corollary*}{Corollary}
\newtheorem{lemma}[theorem]{Lemma}
\newtheorem{proposition}[theorem]{Proposition}
\newtheorem{corollary}[theorem]{Corollary}
\newtheorem{fact}[theorem]{Fact}

\newtheorem{claim}[theorem]{Claim}

\theoremstyle{definition}
\newtheorem{definition}[theorem]{Definition}

\newtheorem{remark}[theorem]{Remark}

\newcommand{\B}{\mathcal{B}}
\newcommand{\C}{\mathcal{C}}

\DeclareMathOperator{\Orb}{Orb}
\DeclareMathOperator{\Tr}{Tr}

\DeclareMathOperator{\im}{im}

\DeclareMathOperator{\supp}{supp}

\DeclareMathOperator{\Iso}{Iso}

\newcommand{\cstar}{\ensuremath{\mathrm{C}^{*}}}
\newcommand{\ch}{\ensuremath{\mathds{1}}}
\newcommand{\Ko}{\ensuremath{\mathrm{K}_{0}}}
\newcommand{\Gunit}{\ensuremath{G^{(0)}}}
\newcommand{\Hunit}{\ensuremath{H^{(0)}}}

\begin{document}

\title[A dichotomy for groupoid \cstar-algebras]{A dichotomy for groupoid \cstar-algebras}
\author[T.\ Rainone and A.\ Sims]{Timothy Rainone and Aidan Sims}
\date{\today}
\address{School of Mathematics and Statistics, Arizona State University, Tempe, Arizona, 85287}
\email{trainone@asu.edu}
\address{School of Mathematics and Applied Statistics, University of Wollongong, NSW 2522, Australia}
\email{asims@uow.edu.au}

\date{\today}
\thanks{This research was supported by the Australian Research Council grant DP150101595.
This research project was initiated when the authors were both at the Centre de Recerca
Mathematica, Universitat Aut\'onoma de Barcelona as part of the Intensive Research
Program \emph{Operator algebras: dynamics and interactions} in 2017, and the work was
significantly supported by the research environment and facilities provided there.}

\begin{abstract}
We study the finite versus infinite nature of $C^*$-algebras arising from \'etale
groupoids. For an ample groupoid $G$, we relate infiniteness of the reduced $C^*$-algebra
of $G$ to notions of paradoxicality of a K-theoretic flavor.  We construct a pre-ordered
abelian monoid $S(G)$ which generalizes the type semigroup introduced by R{\o}rdam and
Sierakowski for totally disconnected discrete transformation groups. This monoid reflects
the finite/infinite nature of the reduced groupoid $C^*$-algebra of $G$. If $G$ is ample,
minimal, and topologically principal, and $S(G)$ is almost unperforated we obtain a
dichotomy between stable finiteness and pure infiniteness for the reduced $C^*$-algebra
of $G$.
\end{abstract}

\maketitle

\section{Introduction}

The groupoid \cstar-algebra construction has been a very fruitful unifying notion in the
theory of operator algebras since Renault's pioneering monograph~\cite{Ren}. Groupoid
\cstar-algebras include all group \cstar-algebras, all crossed products of commutative
\cstar-algebras by actions and partial actions of groups, inverse-semigroup
\cstar-algebras, AF algebras, and the various Cuntz--Krieger constructions. Even in the
seemingly restrictive case of ample groupoids it is known that every Kirchberg algebra
that satisfies the Universal Coefficient Theorem is Morita equivalent to the
\cstar-algebra associated to a Hausdorff ample groupoid (see~\cite{Sp1}). Groupoids and
associated operator algebras have also been used by Connes and others as models for
noncommutative topological spaces.

Here we study, for a large class of \'etale groupoids, the notions of finiteness,
infiniteness, and proper infiniteness, the latter expressed in terms of paradoxical
decompositions. Tarski's alternative theorem establishes, for discrete groups, the
dichotomy between amenability and paradoxicality. This divide carries over to geometric
operator algebras. For example, R{\o}rdam and Sierakowski showed \cite{RS} that if a
discrete group $\Gamma$ acts on itself by left-translation, the Roe algebra
$C(\beta\Gamma)\rtimes_{\lambda}\Gamma$ is properly infinite if and only if $\Gamma$ is
paradoxical, which is equivalent to the non-amenability of $\Gamma$~\cite{RS}. In the
W$^*$-setting, we know that all projections in a II$_1$ factor are finite and that the
ordering of Murray-von-Neumann subequivalence is determined by a unique faithful normal
tracial state. By contrast, type~III factors admit no traces since all non-zero
projections therein are properly infinite. The analogous dichotomy fails for
\cstar-algebras in general as shown by R{\o}rdam's construction~\cite{Ro1} of a unital,
simple, separable and nuclear algebra which is neither stably finite or purely infinite
(the \cstar-algebraic analog of type~III). It is still not known whether the stably
finite/purely infinite dichotomy holds for \cstar-algebras generated by projections. Here
we give a partial answer to this question for \cstar-algebras associated to ample
groupoids.

The motivation for such a dichotomy comes from Elliott's program of classification of
simple, separable, nuclear \cstar-algebras by K-theoretic invariants. In the stably
finite case the Elliott program has seen stunning advances in recent years
\cite{ElliottGongLinNiu, GongLinNiu, RordamYellowBook, TWW, WinterTAF}. In the purely
infinite setting, complete classification was achieved in the groundbreaking results of
Kirchberg~\cite{Kir} and Phillips~\cite{Phi}: all Kirchberg algebras (unital, simple,
separable, nuclear, and purely infinite) satisfying the Universal Coefficient Theorem
(UCT) are classified by their K-theory. For this reason \cstar-theorists have sought to
determine when various \cstar-constructions yield purely infinite simple algebras. For
instance, strong and local boundary actions or certain filling properties displayed in
dynamical systems give rise to purely infinite crossed products \cite{AD2, KiKu, LS,
Sp2}. For groupoids, there is a locally contracting property, introduced by
Anantharaman-Delaroche \cite{AD} that guaranteed pure infiniteness of the associated
\cstar-algebra. This has recently been used, for example, to obtain sufficient conditions
for pure infiniteness of topological-graph \cstar-algebras~\cite{HuiLi}.

The notion of paradoxicality, which undergirds Tarki's theorem~\cite{Wa}, is paramount in
distinguishing the finite from the infinite. This notion was cast in a \cstar-algebraic
framework by Kerr and Nowak~\cite{KN}, R{\o}rdam and Sierakowski~\cite{RS}, and by the
first author in~\cite{Ra2}. The authors of~\cite{RS} built a type semigroup $S(X,\Gamma)$
from an action of a discrete group on the Cantor set and subsequently tied pure
infiniteness of the resulting reduced crossed product to the absence of states on this
semigroup.They prove that if a countable, discrete, and exact group $\Gamma$ acts
continuously and freely on the Cantor set $X$, and the pre-ordered semigroup
$S(X,\Gamma)$ is almost unperforated,  then the following are equivalent: (i) The reduced
crossed product  $C(X)\rtimes_{\lambda}\Gamma$ is purely infinite; (ii)
$C(X)\rtimes_{\lambda}\Gamma$ is traceless; and (iii) $S(X,\Gamma)$ is purely infinite
(that is $2\theta\leq\theta$ for every $\theta\in S(X,\Gamma)$). Inspired by their work,
the first author extended these results to noncommutative \cstar-dynamical systems
$(A,\Gamma)$ by constructing an analogous noncommutative type semigroup
$S(A,\Gamma)$~\cite{Ra2}. In this paper, we generalise this work in a different
direction, to the setting of ample \'{e}tale groupoids. That is, we associate to every
\'{e}tale ample groupoid $G$ a pre-ordered abelian monoid $S(G)$, constructed as an
appropriate quotient of the additive monoid of compactly supported locally constant
integer valued functions on the unit space. We prove that this monoid is an invariant for
equivalence of groupoids, and also that it is isomorphic to the type semigroup
$S(X,\Gamma)$ when $G$ is the transformation groupoid $G=X\rtimes\Gamma$.

Our two main theorems explore the relationship between the nature of $S(G)$ and the
structure of $\cstar_r(G)$. Our strongest results are for the situation where $G$ is
minimal, which is a necessary condition for simplicity of $\cstar_r(G)$. In
Theorem~\ref{theoremsfcnp} we characterize stable finiteness of the reduced
$\cstar$-algebras of minimal groupoids with totally disconnected unit spaces by means of
the finite nature of $S(G)$ and also by means of a coboundary condition. In the amenable
case we recover quasidiagonality. That a certain coboundary condition is equivalent to
stable finiteness is reminiscent of the work of Pimsner~\cite{Pimsner} and
Brown~\cite{BrownAFE} and also appears in noncommutative \cstar-systems~\cite{CaHS, Ra1,
RaSc}. In Theorem~\ref{pitheorem} we establish that, again for minimal groupoids with
totally disconnected unit space, if every element of $S(G)$ is properly infinite, then
$\cstar(G)$ is purely infinite; moreover these conditions are equivalent if $S(G)$ is
almost unperforated.

We deduce a dichotomy for a large class of groupoids: if $\Gunit$ is totally
disconnected, $S(G)$ is almost unperforated, and $G$ is minimal, then $\cstar_r(G)$ is
either stably finite or purely infinite.

As we were preparing the final version of this article, the article \cite{BL} was posted
on the arXiv. In Sections 4~and~5 of that paper, B\"onicke and Li have independently
obtained a number of the results in this paper, for groupoids with compact unit spaces.
We have included remarks discussing the relationships between our results and theirs
where relevant throughout the paper.

\medskip

The paper is organized as follows. We begin in Section~\ref{sec:background} by reviewing
the necessary concepts, definitions, and basic results surrounding the theory of
groupoids and their algebras. In Section~\ref{sec:paradoxicality} we study notions of
paradoxicality displayed in ample groupoids. We construct infinite reduced groupoid
\cstar-algebras and show that stable finiteness is a natural obstruction to paradoxical
behavior. Section~\ref{sec:minimal} deals with minimal groupoids, and provides a
K-theoretic description of minimality for ample groupoids. In Section~\ref{sec:SG} we
associate to every ample groupoid $G$ a type semigroup $S(G)$ and show that it reflects
any paradoxical phenomena present in the groupoid. We show that both isomorphism and
equivalence of ample groupoids preserves the type semigroup. In Section~\ref{sec:stabfin}
we establish our characterization (Theorem~\ref{theoremsfcnp}) of stable finiteness for
the reduced \cstar-algebras of minimal ample groupoids with compact unit space, and
extend this to non-compact unit spaces using the invariance of the type semigroup under
groupoid equivalence. In Section~\ref{sec:pi} we consider the purely infinite situation.
We pin down a necessary condition on $G$ for pure infiniteness of $\cstar_r(G)$, and
prove that if the type semigroup is almost unperforated and the groupoid is minimal and
topologically principal, then pure infiniteness of the type semigroup characterizes pure
infiniteness of the \cstar-algebra (Theorem~\ref{pitheorem}). In particular, we obtain a
dichotomy for the \cstar-algebras of ample minimal groupoids whose type semigroups are
almost unperforated. In section~\ref{sec:reconcile}, we reconcile our definition of the
type semigroup $S(G)$ with previous constructions for group actions and for higher-rank
graphs. Finally in Section~\ref{sec:applications}, we explore applications of our ideas
to orbit equivalence, to $n$-filling groupoids, and to zero-dimensional topological
graphs.

\smallskip

The first author would like to express his gratitude to Jack Spielberg for his many
suggestions and careful comments. A word of thanks to Christopher Schafhauser for
meaningful discussions and answered inquiries. The second author thanks Adam Sierakowski
for many helpful conversations.

\section{Preliminaries and Notation}\label{sec:background}

In this section we provide a brief introduction to the theory of groupoids,  topological
groupoids, and their reduced \cstar-algebras. We shall be mainly interested in the
\'{e}tale case. Experts are welcome to move ahead to the next section. There are a
variety of good resources on the subject, for example Patterson's book~\cite{Pat}, or the
work of Reneault~\cite{Ren}.

\subsection{Groupoids.} \label{introgroupoids}

A \emph{groupoid} is a non-empty set $G$ satisfying the following:
\begin{enumerate}
\item[G1.] There is a distinguished subset $\Gunit\subseteq G$, called the unit space.
    Elements of $\Gunit$ are called units.
\item[G2.] There are maps $r,s:G\rightarrow\Gunit$ satisfying $r(u)=s(u)=u$ for all
    $u\in \Gunit$. These maps are called the range and source maps respectively.
\item [G3.] Setting $G^{(2)}=\{(\alpha,\beta)\ |\ \alpha, \beta\in G,\
    s(\alpha)=r(\beta)\}$, there is a `law of composition'
\[m:G^{(2)}\longrightarrow G,\quad m(\alpha,\beta)=\alpha\beta\]
that satisfies
\begin{enumerate}
\item [(i)] $r(\alpha\beta)=r(\alpha)$, $s(\alpha\beta)=s(\beta)$, for all
    $(\alpha,\beta)\in G^{(2)}$.
\item[(ii)] If $(\alpha,\beta)$ and $(\beta,\gamma)$ belong to $G^{(2)}$, then
    $(\alpha,\beta\gamma)$ and $(\alpha\beta,\gamma)$ also are in $G^{(2)}$ and
    $(\alpha\beta)\gamma=\alpha(\beta\gamma)$.
\item[(iii)] For every $\alpha\in G$, $r(\alpha)\alpha=\alpha=\alpha s(\alpha)$
    (note that $(r(\alpha),\alpha)$ and $(\alpha, s(\alpha))$ are in $G^{(2)}$).
\end{enumerate}
\item[G4.] For every $\alpha\in G$ there is an `inverse' $\alpha^{-1}\in G$
    (necessarily unique) such that $(\alpha,\alpha^{-1})$ and $(\alpha^{-1},\alpha)$
    are in $G^{(2)}$ and such that $\alpha\alpha^{-1}=r(\alpha)$,
    $\alpha^{-1}\alpha=s(\alpha)$.
\end{enumerate}
It follows from definitions that $r(\alpha^{-1})=s(\alpha)$, $s(\alpha^{-1})=r(\alpha)$,
and that the map $\iota(\alpha)=\alpha^{-1}$ is an involutive bijection of $G$.

The following notation is common in the groupoid literature. If $A,B\subseteq G$, then
\[AB=\{\alpha\beta\ |\ \alpha\in A, \beta\in B, r(\beta)=s(\alpha)\}=m\left( (A\times B)\cap G^{(2)}\right).\]
Some special cases arise frequently; for instance, if $E\subseteq G$ and $U\subseteq \Gunit$
then
\begin{align*}
EU=&\{\alpha\beta\ |\ \alpha\in E, \beta\in U, s(\alpha)=r(\beta)\}=\{\alpha\beta\ |\ \alpha\in E, \beta\in U, s(\alpha)=\beta\}\\=&\{\alpha s(\alpha)\ |\ \alpha\in E,  s(\alpha)\in U\}=\{\alpha\ |\ \alpha\in E,  s(\alpha)\in U\}=E\cap s^{-1}(U),
\end{align*}
and in this context we will write
\[U_E:=r(EU)=\{r(\alpha)\ |\ \alpha\in E, s(\alpha)\in U\}\subseteq \Gunit.\]
As a special case, if $u$ is a unit in $G$, then
\[G_u:=Gu=\{\alpha\in G\ |\ s(\alpha)=u\},\quad G^{u}:=uG=\{\alpha\in G\ |\ r(\alpha)=u\},\quad G_u^u:=G_u\cap G^u.\]
The \emph{isotropy subgroupoid} is defined by $\Iso(G):=\bigcup_{u\in \Gunit}G_u^u$.

A map $\varphi:G\rightarrow H$ between groupoids is a homomorphism of groupoids if
\[(\alpha,\beta)\in G^{(2)}\Longrightarrow (\varphi(\alpha),\varphi(\beta))\in H^{(2)}\quad\text{and}\quad \varphi(\alpha\beta)=\varphi(\alpha)\varphi(\beta).\]
It is not difficult to show that such a homomorphism satisfies
\[\varphi(\Gunit)\subseteq\Hunit,\quad \varphi\circ r_G=r_H\circ\varphi,\quad \varphi\circ s_G=s_H\circ\varphi,\quad \varphi\circ\iota_G=\iota_H\circ\varphi.\]
If $\varphi$ is bijective, then the inverse map $\varphi^{-1}:H\rightarrow G$ is easily
seen to be a groupoid homomorphism, and such a $\varphi$ defines a groupoid isomorphism.
If $\varphi$ is an isomorphism, then $\varphi(\Gunit)=\Hunit$.

We are mainly interested in groupoids that admit a topology. A \emph{topological
groupoid} is a groupoid $G$ with a topology for which $G^{(2)}\subseteq G\times G$ is
closed (automatic when $G$ is Hausdorff), $m:G^{(2)}\rightarrow G$ is continuous, and
$\iota:G\rightarrow G$ is continuous. It follows that the range and source maps are again
continuous, and if $G$ is Hausdorff, it follows that $\Gunit$ is closed. An isomorphism
of topological groupoids is an isomorphism of groupoids that is also a homeomorphism.

A topological groupoid $G$ is called \emph{minimal} if $\overline{r(G_u)}=\Gunit$ for
every unit $u\in\Gunit$,  and is said to be \emph{topologically principal} or
\emph{effective} if $\Iso(G)^{\circ}=\Gunit$.

A topological groupoid $G$ is called \emph{\'{e}tale} if the maps $r,s:G\rightarrow G$
are local homeomorphisms. In this setting every $\alpha\in G$ has an open neighborhood
$E\subseteq G$ for which $r(E), s(E)$ are open subsets of $G$, and $r|_{E}:E\rightarrow
r(E)$, $s|_E:E\rightarrow s(E)$ are homeomorphisms. Such a set $E$ is called an
\emph{open bisection}. A subset of an open bisection is simply called a \emph{bisection}.
One shows that the unit space $\Gunit$ of an \'{e}tale groupoid is open and closed in
$G$. If $G$ is a locally compact Hausdorff groupoid, then $G$ is \'{e}tale if and only if
there is a basis for the topology on $G$ consisting of open bisections with compact
closure. A topological groupoid is called \emph{ample} if $G$ has a basis of compact open
bisections. If $G$ is locally compact, Hausdorff, and \'{e}tale, then  $G$ is ample if
and only if $\Gunit$ is totally disconnected (see Proposition 4.1 in~\cite{Ex1}).

When $G$ is \'{e}tale we will denote by $\B$ the collection of open bisections, and write
$\C\subseteq\B$ for the subcollection of all compact open bisections. It can be shown that
both $\B$ and $\C$ are closed under inversion, multiplication, and taking intersections.
Note that any open (compact open) $U\subseteq\Gunit$ belongs to $\B$ ($\C$).
%We note the following standard algebraic properties of bisections:
%\begin{enumerate}
%\item[] $E\in \B\  (\C)\Longrightarrow E^{-1}\in\B\  (\C)$,
%\item[] $E,F\in \B\  (\C)\Longrightarrow EF\in\B\  (\C)$,
%\item[] $E,F\in \B\  (\C)\Longrightarrow E\cap F\in\B\ (\C)$.
%\end{enumerate}
The following facts will surface regularly throughout our work. If $E$ is a bisection
in $G$, then using the fact that $r$ and $s$ are injective on $E$, we get
\[E^{-1}E=\{\alpha^{-1}\beta\ |\ \alpha,\beta\in E, r(\beta)=s(\alpha^{-1})=r(\alpha)\}=\{\alpha^{-1}\alpha\ |\ \alpha\in E\}=s(E),\]
\[EE^{-1}=\{\alpha\beta^{-1}\ |\ \alpha,\beta\in E, s(\alpha)=r(\beta^{-1})=s(\beta)\}=\{\alpha\alpha^{-1}\ |\ \alpha\in E\}=r(E).\]
Moreover, if $E$ and $F$ are disjoint bisections in $G$, then
$E^{-1}F\cap\Gunit=\emptyset=EF^{-1}\cap\Gunit$. This holds because for any
$\alpha,\beta\in G$, $\alpha^{-1}\beta\in\Gunit$ implies that $\alpha=\beta$. Finally we
note that if $E,U$ belong to $\B$ ($\C$) with $U\subseteq\Gunit$, then
$EU=(s|_{E})^{-1}(U)$ and $U_E=r(EU)$ also belong to $\B$ ($\C$), because $s|_{E}$ is a
homeomorphism and $r$ is an open map.

For convenience we shall henceforth assume that \emph{all topological groupoids are
locally compact, second countable, and Hausdorff.}

The theory of topological groupoids offers a unification of several constructions. One
motivating special case is that of a topological dynamical system. A \emph{transformation
group} is a pair $(X,\Gamma)$ where $X$ is a locally compact Hausdorff space, $\Gamma$ is
a locally compact Hausdorff group, and $\Gamma\curvearrowright X$ acts continuously by
homeomorphisms. Endowed with the product topology, $G=\Gamma\times X$ becomes a locally
compact Hausdorff groupoid where the unit space $\Gunit:=\{(e,x)\ |\ x\in X\}\cong X$ can
be identified with $X$, and for $t,t'\in\Gamma$, $x,y\in X$
\[s(t,x)=(e,x)\cong x,\quad r(t,x)=(e,t.x)\cong t.x,\quad m((t',y),(t,x))=(t't,x),\quad\text{if}\quad t.x=y.\]
This groupoid is often called the \emph{transformation groupoid} and is denoted by
$G=X\rtimes\Gamma$. If $\Gamma$ is discrete then $X\rtimes\Gamma$ is \'{e}tale, and if,
moreover, $X$ is totally disconnected then $X\rtimes\Gamma$ is ample. In the \'{e}tale
setting one can show that $X\rtimes\Gamma$ is minimal if and only if the action
$\Gamma\curvearrowright X$ is minimal ($X$ admits no non-trivial $\Gamma$-invariant
closed subspaces), and $X\rtimes\Gamma$ is topologically principal if and only if the
action $\Gamma\curvearrowright X$ is topologically free (every point in $X$ has
nowhere-dense stabilizer). In this dynamical setting $\Gamma$ naturally acts on functions
defined on $X$ through the formula $t.f(x)=f(t^{-1}.x)$, where $f:X\rightarrow\mathbb{C}$
is any function. Generalizing this to the groupoid setting we arrive at the following
definition.

\begin{definition}
Let $G$ be an \'{e}tale groupoid and $E$ a bisection. If $f:\Gunit\rightarrow\mathbb{C}$
is a function, define $Ef:\Gunit\rightarrow\mathbb{C}$ by
$$
Ef(u)=
\begin{cases}
f(s(\alpha)) & \text{ if } \exists \alpha\in E\ \text{with}\ r(\alpha)=u,\\
0 & \text{ otherwise}.
\end{cases}
$$
\end{definition}
Note that $Ef$ is well defined because $r|_{E}:E\rightarrow\Gunit$ is injective. We list
a few facts that will surface periodically in our work below. We are always assuming that
$G$ is \'{e}tale. The first easily follows from definitions.

\begin{fact}\label{factlinear}
If $E$ is a bisection, $f,g:\Gunit\rightarrow\mathbb{C}$ functions, and $z\in\mathbb{C}$,
then $$E(zf+g)=zEf+Eg.$$
\end{fact}

\begin{fact}\label{factchar}
If $E\subseteq G$ is a bisection and $U\subseteq\Gunit$, then
$E\ch_U=\ch_{r(EU)}=\ch_{U_{E}}$.
\end{fact}
\begin{proof}
Definitions imply $U_E=r(EU)=\{r(\alpha)\ |\ \alpha\in E,  s(\alpha)\in U\}$, so
\[\ch_{U_E}(u)=
\begin{cases}
1 & \text{ if } \exists\alpha\in E\ \text{with}\ s(\alpha)\in U,\ r(\alpha)=u,\\
0 & \text{ otherwise}.
\end{cases}\]
whereas
\begin{align*}
E\ch_{U}(u)&=
\begin{cases}
\ch_U(s(\alpha)) & \text{ if } \exists\alpha\in E\ \text{with}\ r(\alpha)=u,\\
0 & \text{ otherwise}.
\end{cases}\\
&=
\begin{cases}
1 & \text{ if } \exists\alpha\in E\ \text{with}\ s(\alpha)\in U,\ r(\alpha)=u,\\
0 & \text{ otherwise}.
\end{cases}
\end{align*}
Thus $\ch_{U_E}(u)=E\ch_{U}(u)$ for every $u\in\Gunit$.
\end{proof}

\begin{fact}\label{factcts}
If $f:\Gunit\rightarrow\mathbb{C}$ is continuous and $E$ is a closed and open bisection,
then $Ef:\Gunit\rightarrow\mathbb{C}$ is continuous.
\end{fact}
\begin{proof}
Let $u, (u_n)_n\in\Gunit$ with $(u_n)_n\longrightarrow u$. Suppose first that there is an
$\alpha\in E$ with $r(\alpha)=u$. Then $u\in r(E)$ which is open, so $u_n$ belongs to
$r(E)$ for large enough $n$. Say $r(\alpha_n)=u_n$, for some $\alpha_n\in E$. Now
$r|_{E}$ is a homeomorphism, so $(\alpha_n)_n\longrightarrow\alpha$ which implies
$(s(\alpha_n))_n\longrightarrow s(\alpha)$. Therefore
\[\left(Ef(u_n)\right)_n=\left(f(s(\alpha_n))\right)_n\longrightarrow f(s(\alpha))=Ef(u).\]

Next suppose that $u\notin r(E)$, which means that $Ef(u)=0$. Then $u\in \Gunit\setminus
r(E)$ which is again open, so that $u_n\notin r(E)$ for all large $n$. In this case
$Ef(u_n)=0$ for all large $n$.
\end{proof}

\subsection{The reduced \texorpdfstring{\cstar}{C*}-algebra \texorpdfstring{$\cstar_r(G)$}{C*(G)}.}
We briefly describe the construction of the reduced \cstar-algebra of a locally compact,
Hausdorff, \'{e}tale groupoid $G$. Fix such a $G$ and consider the complex linear space
$C_c(G)$ of compactly supported complex-valued functions on $G$. Convolution and
involution defined by
\begin{align*}
f\cdot g(\gamma)&=\sum_{\alpha\beta=\gamma}f(\alpha)g(\beta),\quad\gamma\in G\\f^{*}(\alpha)&=\overline{f(\alpha^{-1})},\quad\alpha\in G
\end{align*}
give $C_c(G)$ the structure of a complex $\ast$-algebra. Is is important to note that
there is a natural inclusion $C_c(\Gunit)\hookrightarrow C_c(G)$ of $\ast$-algebras.
Also, if $E$ and $F$ are compact open bisections, then the characteristic functions
$\ch_E, \ch_F\in C_c(G)$ and
\[\ch_E^{*}=\ch_{E^{-1}}\quad\text{and}\quad\ch_E\ch_F=\ch_{EF}.\]
The $\ast$-algebra $C_c(G)$ is then represented on Hilbert spaces as follows: for a unit
$u\in\Gunit$ define
\[\pi_u:C_c(G)\rightarrow\mathbb{B}(\ell^{2}(G_u))\quad\text{as}\quad\pi_u(f)(\xi)(\gamma)=f\cdot\xi(\gamma)=\sum_{\alpha\beta=\gamma}f(\alpha)\xi(\beta).\]
It is verified that $\pi_u$ is a representation of the $\ast$-algebra $C_c(G)$, and that
the direct sum $\pi_r:=\bigoplus_{u\in\Gunit}\pi_u$ is faithful. It follows that
$\|f\|_r:=\|\pi_r(f)\|$ is a \cstar-norm on $C_c(G)$ and we may define the reduced
\cstar-algebra of the groupoid $G$ as
\[\cstar_r(G):=\overline{C_c(G)}^{\|\cdot\|_r}.\]

The inclusion $C_c(\Gunit)\hookrightarrow\cstar_r(G)$ has a partial `inverse', namely the
conditional expectation. More precisely, the restriction map $C_c(G)\rightarrow
C_c(\Gunit)$, $f\mapsto f|_{\Gunit}$ extends continuously to a faithful conditional
expectation $\mathbb{E}:\cstar_r(G)\rightarrow C_0(\Gunit)$.

Given a discrete transformation group $(X,\Gamma)$, the \cstar-algebra  obtained from the
resulting transformation  groupoid $X\rtimes\Gamma$ is a familiar construction; the
reduced \cstar-crossed product. In fact,
$$\cstar_{r}(X\rtimes\Gamma)\cong C_0(X)\rtimes_{r}\Gamma.$$

\subsection{K-theory and traces.} \label{Kth&traces}
Recall that if $X$ is the Cantor set, $\Ko(C(X))$ is order isomorphic to
$C(X,\mathbb{Z})$ via the dimension map $\dim:\Ko(C(X))\rightarrow C(X,\mathbb{Z})$ given
by $\dim([p]_0)(x)=\Tr(p(x))$, where $p$ is a projection over the matrices of $C(X)$;
$M_n(C(X))\cong C(X;\mathbb{M}_{n})$, and $\Tr$ denotes the standard (non-normalized)
trace on $\mathbb{M}_{n}$. It is clear that under this isomorphism $\dim
([\ch_E]_0)=\ch_E$ where $E\subseteq X$ is a compact open subset.

If $G$ is an \'{e}tale groupoid with unit space $\Gunit$ homeomorphic to a Cantor set,
the inclusion $\iota:C(\Gunit)\hookrightarrow\cstar_r(G)$ induces a positive group
homomorphism $$\Ko(\iota):\Ko(C(\Gunit))\longrightarrow\Ko(\cstar_r(G))$$ that satisfies
$\Ko(\iota)([\ch_{s(E)}]_0)=\Ko(\iota)([\ch_{r(E)}]_0)$ for any compact open bisection
$E$. This holds because the projections $\iota(\ch_{s(E)})$ and $\iota(\ch_{r(E)})$ are
Murray-von-Neumann equivalent in $\cstar_r(G)$. Indeed, $\ch_E\in
C_c(G)\subseteq\cstar_r(G)$ and
\[\ch_E^*\ch_E=\ch_{E^{-1}}\ch_E=\ch_{E^{-1}E}=\ch_{s(E)},\]
\[\ch_E\ch_E^{-1}=\ch_E\ch_{E^{-1}}=\ch_{EE^{-1}}=\ch_{r(E)}.\]

The following elementary fact will be useful. Let $A$ and $B$ be \cstar-algebras with $B$
stably finite, and suppose $\varphi:A\hookrightarrow B$ is an embedding. The induced map
on K-theory $\Ko(A)\rightarrow\Ko(B)$ is faithful. For if $p$ is a projection in
$\mathbb{M}_n(A)$ for some $n$, then
\[\Ko(\varphi)([p]_0)=0\Longrightarrow[\varphi(p)]_0=0\Longrightarrow\varphi(p)=0\Longrightarrow p=0,\]
where the second implication follows from the fact that $B$ is stably finite.

Traces on $\cstar_r(G)$ are often obtained via invariant measures on the unit space. We describe these.

\begin{definition}
Let $G$ be an \'{e}tale groupoid. A regular Borel measure $\mu$ on $\Gunit$ is said to be
$G$-invariant if, $\mu(s(E))=\mu(r(E))$ for every open bisection $E$.
\end{definition}

For such a groupoid $G$ and invariant probability measure $\mu$, we obtain a tracial state
$\tau_{\mu}$ on $\cstar_r(G)$ by composing the conditional expectation
$\mathbb{E}:\cstar_r(G)\rightarrow C_0(\Gunit)$ with integration against $\mu$;
$I_{\mu}:C_0(\Gunit)\rightarrow\mathbb{C}$, $f\mapsto\int_{\Gunit}fd\mu$:
\[\tau_{\mu}:=I_{\mu}\circ\mathbb{E}:\cstar_r(G)\longrightarrow\mathbb{C}\quad \tau_{\mu}(a)=\int_{\Gunit}\mathbb{E}(a)d\mu=\int_{\Gunit}a(u)d\mu(u).\]
Since the expectation is faithful, the above trace $\tau_{\mu}$ is faithful provided the
measure $\mu$ has full support. Finally, if $G$ is such a groupoid with the added
assumption that $G$ is principal, then every trace on $\cstar_r(G)$ arises in this way.

\section{Paradoxical Groupoids}\label{sec:paradoxicality}

We begin by studying notions of paradoxicality and paradoxical decompositions in
\'{e}tale and ample groupoids which give rise to infinite reduced groupoid
\cstar-algebras. Recall that a projection $p$ in a \cstar-algebra $A$ is infinite if
there is a partial isometry $v\in A$ with $v^*v=p$ and $vv^*<p$, that is, $p$ is
Murray-von Neumann equivalent to a proper subprojection of itself. A \cstar-algebra $A$
is infinite if it admits an infinite projection.

\begin{definition}\label{defnparadox}
Let $G$ be an \'{e}tale and ample groupoid, and suppose $A\subseteq \Gunit$ is a non-empty
compact open subset of the unit space. We say that $A$ is \emph{paradoxical} if there are
bisections $E_1,\dots, E_n$ in $\C$ satisfying
\begin{equation}\label{paradox} \ch_{A}\leq\sum_{i=1}^{n}\ch_{s(E_i)},\quad\text{and}\quad \sum_{i=1}^{n}\ch_{r(E_i)}<\ch_{A}.
\end{equation}

We will say that $G$ is \emph{paradoxical} if there is some paradoxical $A\subseteq\Gunit$.

\end{definition}

A remark is in order. The notation $f<g$, for functions $f,g\in C_c(\Gunit)$, means that
$f\leq g$ and $f\neq g$.

\begin{proposition}
Let $G$ be an \'{e}tale and ample groupoid. If $G$ is paradoxical, then the reduced
groupoid \cstar-algebra $\cstar_r(G)$ is infinite.
\end{proposition}

\begin{proof}
Suppose the compact open $\emptyset\neq A\subseteq\Gunit$ is paradoxical. Let $E_1,\dots
E_n$ be the compact open bisections that satisfy~\ref{paradox}. The inequality
$\sum_{i=1}^{n}\ch_{r(E_i)}<\ch_{A}$ implies that $r(E_i)\cap r(E_j)=\emptyset$ for
$i\neq j$. It naturally follows that the bisections $E_i$ are  also disjoint. Moreover,
since $\ch_{A}\leq\sum_{i=1}^{n}\ch_{s(E_i)}$ we know that
$A\subseteq\bigcup_{i=1}^{n}s(E_i)$. There are compact open subsets $A_1,\dots, A_n\subseteq
A$ that satisfy $A_i\subseteq s(E_i)$ and $\bigsqcup_{i=1}^{n}A_i=A$. This is seen by
inductively defining
\[A_1:=A\cap s(E_1),\ A_2:=\left(A\cap s(E_2)\right)\setminus A_1,\dots, A_n:=\left(A\cap s(E_n)\right)\setminus\left(\bigcup_{i=1}^{n-1}A_i\right).\]
Now let $F_i:=(s|_{E_i})^{-1}(A_i)\subseteq E_i$ for $i=1,\dots, n$. Note that the $F_i$
are again disjoint compact open bisections and that $s(F_i)=A_i$. Observe
\[F_i^{-1}F_j=\{\alpha^{-1}\beta\ |\ \alpha\in F_i, \beta\in F_j, r(\alpha)=s(\alpha^{-1})=r(\beta)\}=
\begin{cases}
s(F_i)=A_i & \text{ if } i=j,\\
\emptyset & \text{ if } i\neq j.
\end{cases}\]
\[F_iF_j^{-1}=\{\alpha\beta^{-1}\ |\ \alpha\in F_i, \beta\in F_j, s(\alpha)=r(\beta^{-1})=s(\beta)\}=
\begin{cases}
r(F_i) & \text{ if } i=j,\\
\emptyset & \text{ if } i\neq j.
\end{cases}.\]
Therefore if we set
\[v:=\sum_{i=1}^{n}\ch_{F_i}\ \text{in}\ C_c(G)\subseteq\cstar_r(G),\]
we obtain
\begin{align*}\textstyle
v^*v
    &= \textstyle \Big(\sum_{i=1}^{n}\ch_{F_i}\Big)^*\Big(\sum_{j=1}^{n}\ch_{F_j}\Big)
     = \sum_{i,j=1}^{n}\ch_{F_i}^*\ch_{F_j}\\
    &= \textstyle \sum_{i,j=1}^{n}\ch_{F_i^{-1}F_j}
     = \sum_{i=1}^{n}\ch_{s(F_i)}
     = \sum_{i=1}^{n}\ch_{A_i}=\ch_{A}.
\end{align*}
whereas
\begin{align*}
vv^*
    &= \textstyle \Big(\sum_{i=1}^{n}\ch_{F_i}\Big)\Big(\sum_{j=1}^{n}\ch_{F_j}\Big)^{*}
     = \sum_{i,j=1}^{n}\ch_{F_i}\ch_{F_j}^{*}\\
    &= \sum_{i,j=1}^{n}\ch_{F_iF_j^{-1}}
     =\sum_{i=1}^{n}\ch_{r(F_i)}
     \leq \sum_{i=1}^{n}\ch_{r(E_i)}<\ch_{A}.
\end{align*}
The projection $\ch_A$ is therefore infinite whence $\cstar_r(G)$ is infinite.
\end{proof}

Parodoxicality carries the connotation of `duplication of sets', so we revisit the ideas
explored in~\cite{KN} and ~\cite{Ra2} and define, in the groupoid setting, a notion of
paradoxical decomposition with a covering multiplicity.

\begin{definition}
Let $G$ be an \'{e}tale and ample groupoid and suppose $A\subseteq \Gunit$ is a non-empty
compact open subset of the unit space. With $k>l>0$ positive integers, we say that $A$ is
\emph{$(k,l)$-paradoxical} if there are bisections $E_1,\dots, E_n$ in $\C$ satisfying
\begin{equation}\label{klparadox} k\ch_{A}\leq\sum_{i=1}^{n}\ch_{s(E_i)},\quad\text{and}\quad \sum_{i=1}^{n}\ch_{r(E_i)}\leq l\ch_{A}.
\end{equation}

We call $A$ \emph{properly paradoxical} if it is $(2,1)$-paradoxical.

If $A$ fails to be $(k,l)$-paradoxical for all integers $k>l>0$ then we say that $A$ is
\emph{completely non-paradoxical.}

If every compact open $A\subseteq \Gunit$ is completely non-paradoxical we say that $G$ is
completely non-paradoxical.
\end{definition}

\begin{remark}
A related definition of paradoxicality appears as \cite[Definition~4.4]{BL}; our
definitions is slightly more flexible in that we do not require that the sets $s(E_i)$
in~\eqref{klparadox} cover $A$, and we do not insist on any orthogonality amongst the
sets $r(E_i)$. It is clear that if $G$ is $(\mathbb{E}, k, l)$-paradoxical in the sense
of B\"onicke and Li for some $\mathbb{E}, k, l$, then it is $(k,l)$-paradoxical in our
sense.
\end{remark}

It is not surprising that stable finiteness is an obstruction to paradoxicality. The
following result is related to \cite[Proposition~4.8]{BL}, modulo the difference in our
definitions of paradoxicality.

\begin{proposition}\label{sfimpliescnp}
Let $G$ be an \'{e}tale, and ample groupoid. If $\cstar_{r}(G)$ is stably finite then $G$
is completely non-paradoxical.
\end{proposition}
\begin{proof}
Suppose $A\subseteq \Gunit$ is compact open and also $(k,l)$-paradoxical. Let $E_1,\dots
E_n\in\C$ be the bisections that satisfy~\ref{klparadox}. Moreover, let $\iota:
C(\Gunit)\hookrightarrow\cstar_{r}(G)$ denote the canonical inclusion. Since
$\cstar_{r}(G)$  is stably finite,  the induced map on K-theory
\[\mathrm{K}_0(\iota): \mathrm{K}_0(C(\Gunit))\cong C(\Gunit,\mathbb{Z})\longrightarrow \mathrm{K}_{0}(\cstar_{r}(G))\]
is faithful (see section~\ref{Kth&traces}). We compute
\begin{align*}
k\mathrm{K}_0(\iota)(\ch_{A})
    &=\mathrm{K}_0(\iota)(k\ch_{A})
     \leq\mathrm{K}_0(\iota)\Big(\sum_{i=1}^{n}\ch_{s(E_i)}\Big)
     =\sum_{i=1}^{n}\mathrm{K}_0(\iota)(\ch_{s(E_i)})\\
    &=\sum_{i=1}^{n}\mathrm{K}_0(\iota)(\ch_{r(E_i)})
     =\mathrm{K}_0(\iota)\Big(\sum_{i=1}^{n}\ch_{r(E_i)}\Big)
     \leq \mathrm{K}_0(\iota)(l\ch_{A})
     =l\mathrm{K}_0(\iota)(\ch_{A}).
\end{align*}
Writing $x=\mathrm{K}_0(\iota)(\ch_{A})$ we get that $(l-k)x$ belongs to
$\mathrm{K}_{0}(\cstar_{r}(G))^+\cap -\left(\mathrm{K}_{0}(\cstar_{r}(G))^+\right)$,
which is trivial by stable finiteness. Thus $(l-k)x=0$. Again using the fact that
$\cstar_{r}(G)$ is stably finite we get $x=0$, and so $\ch_{A}=0$ since
$\mathrm{K}_0(\iota)$ is faithful. This contradicts the assumption that $A$ is non-empty.
\end{proof}

One of the principal goals of this paper is to characterize stably finite \cstar-algebras
that arise from \'{e}tale groupoids. In this vein we aim to establish a converse to
Proposition~\ref{sfimpliescnp}. What is needed is a technique to pass from complete
non-paradoxicality of a groupoid $G$ to constructing faithful tracial states on
$\cstar_r(G)$. We will achieve this in the minimal setting (Theorem~\ref{theoremsfcnp}).
The next section is devoted to expressing the notion of minimality in a K-theoretic
framework which will best suit this purpose and enable us to construct faithful traces.

\section{Minimal Groupoids}\label{sec:minimal}

Recall that a \cstar-dynamical system $(A,\Gamma,\alpha)$ is said to be minimal if $A$
admits no non-trivial $\Gamma$-invariant ideals. If $A=C_0(X)$, where $X$ is a locally
compact space, and $\alpha$ is induced by a continuous action $\Gamma\curvearrowright X$,
then $\alpha$ is minimal if and only if $X$ admits no non-trivial $\Gamma$-invariant
closed subsets. When $X$ is compact it is not difficult to show that the action is
minimal if and only if the orbit $\Orb(x):=\{t.x\ |\ t\in\Gamma\}$ of every $x\in X$ is
dense in $X$. The analogous statement for \'{e}tale groupoids is natural. Indeed, an
\'{e}tale groupoid $G$ is said to be \emph{minimal} if the `orbit' of every unit is dense
in the unit space, that is, $\overline{r(G_u)}=\Gunit$ for every unit $u\in \Gunit$.

If $(X,\Gamma)$ is a discrete transformation group with $X$ compact, it is not difficult
to show (see Proposition 3.1 in~\cite{Ra3}) that the action is minimal if and only if
the following holds: for every non-empty open set $U\subseteq X$, there are group elements
$t_1,\dots, t_n\in\Gamma$ such that $X=\cup_{j=1}^{n}t_j.U$. The following result is
similar, where group elements are replaced by open bisections.

Recall from section~\ref{introgroupoids} that if $E$ is an open bisection in $G$ and
$U\subseteq\Gunit$ is open, then both
\[EU=\{\alpha\ |\ \alpha\in E,  s(\alpha)\in U\},\quad\text{and}\quad U_E=r(EU)=\{r(\alpha)\ |\ \alpha\in E,  s(\alpha)\in U\}\]
are also open.

\begin{proposition}\label{min1}
Let $G$ be an \'{e}tale groupoid with compact unit space $\Gunit$. The following are
equivalent
\begin{enumerate}
\item [(i)] $G$ is minimal.
\item [(ii)] For every non-empty open $U\subseteq \Gunit$, there are open bisections
    $E_1,\dots, E_n$ such that
\[\bigcup_{j=1}^{n}U_{E_j}=\Gunit.\]
\end{enumerate}
\end{proposition}

\begin{proof}
$(i)\Rightarrow (ii)$: Let $U\subseteq \Gunit$ be open. For each finite collection of open
bisections $\mathcal{E}=\{E_1,\dots, E_m\}\subseteq\B$ we set
\[U_{\mathcal{E}}=\bigcup_{E\in\mathcal{E}}U_{E}.\]
We claim that $\bigcup_{\mathcal{E}}U_{\mathcal{E}}=\Gunit$ where the union runs over all
finite collections $\mathcal{E}$ of open bisections. Suppose the claim is false, then
there is a unit $u\in \Gunit\setminus\cup_{\mathcal{E}}U_{\mathcal{E}}$. Since $r(G_u)$
is dense and  $\cup_{\mathcal{E}}U_{\mathcal{E}}$ is open we know that $r(G_u)\cap
\left(\cup_{\mathcal{E}}U_{\mathcal{E}}\right)\neq\emptyset$. Let $\alpha\in G_u$ with
$r(\alpha)\in U_{\mathcal{F}}$ for some finite collection $\mathcal{F}\subseteq\B$. So
there is a bisection $F\in\mathcal{F}$ with $r(\alpha)\in E_F$. This means that there is
a $\beta\in F$ with $s(\beta)\in U$ and $r(\alpha)=r(\beta)$. Now set
$\gamma=\alpha^{-1}\beta$. It follows that
\[s(\gamma)=s(\beta)\in U,\quad \text{and}\quad r(\gamma)=r(\alpha^{-1})=s(\alpha)=u.\]
Now let $E$ be an open bisection containing $\alpha^{-1}$. Then $H=EF$ is an open
bisection containing $\gamma$, and $u\in U_H$ since $\gamma\in H$ and $r(\gamma)=u$ with
$s(\gamma)\in U$. This contradicts the fact that
$u\notin\cup_{\mathcal{E}}U_{\mathcal{E}}$. The claim is thus proved.

Compactness of $\Gunit$ implies that $\Gunit=\cup_{j=1}^{J}U_{\mathcal{E}_{j}}$ where the
$\mathcal{E}_j$ are finite collections of open bisections. Let
$\mathcal{E}=\cup_{j=1}^{J}\mathcal{E}_j$, which is again a finite collection. Therefore,
\[\Gunit=U_{\mathcal{E}}=\bigcup_{E\in\mathcal{E}}U_{E}.\]

$(ii)\Rightarrow (i)$: Let $u\in \Gunit$, we want to show that
$\overline{r(G_u)}=\Gunit$. If not, set $U=\Gunit\setminus\overline{r(G_u)}$, a non-empty
open set in $\Gunit$. By our assumption there are open bisections $E_1,\dots, E_n$ such
that
\[\bigcup_{j=1}^{n}U_{E_j}=\Gunit.\]
Let $i$ be such that $u\in U_{E_i}$. Then there is an $\alpha$ in $E_i$ with
$s(\alpha)\in U$ and $r(\alpha)=u$. This implies that $\alpha^{-1}\in G_u$ since
$s(\alpha^{-1})=r(\alpha)=u$. However,
\[\overline{r(G_u)}\supseteq r(G_u)\ni r(\alpha^{-1})=s(\alpha)\in U\]
which contradicts the fact that $U\cap\overline{r(G_u)}=\emptyset$.
\end{proof}

We now turn our attention to the ample case. Let $G$ be an \'{e}tale groupoid with
compact and totally disconnected unit space $\Gunit$. If $f:\Gunit\rightarrow\mathbb{Z}$
is a continuous integer-valued function, its range is a finite subset of $\mathbb{Z}$ and
we can therefore express $f$ as a finite sum
\[f=\sum_{k=1}^{l}m_k\ch_{U_k}\]
where $m_k\in\mathbb{Z}$ and the $U_k$ are closed and open subsets of $\Gunit$. If $E$ is
a compact and open bisection in $G$, then combining
Facts~\ref{factlinear},~\ref{factchar}, and~\ref{factcts} we get
\[Ef=E\left(\sum_{k=1}^{l}m_k\ch_{U_k}\right)=\sum_{k=1}^{l}m_k\ch_{r(EU_k)}\]
which is again a continuous integer-valued function on $\Gunit$.

\begin{proposition}\label{min2}
Let $G$ be an \'{e}tale groupoid with compact unit space
$\Gunit$. The following are equivalent.
\begin{enumerate}
\item [(i)] $G$ is minimal.
\item [(ii)] For every non-empty closed and open $U\subseteq \Gunit$, there are clopen bisections $E_1,\dots, E_n$ such that
\[\bigcup_{j=1}^{n}U_{E_j}=\Gunit.\]
\item [(iii)] For every non-zero $f\in C(\Gunit,\mathbb{Z})^+$, there are clopen bisections $E_1,\dots, E_n$ such that
\[\sum_{j=1}^{n}E_jf\geq\ch_{\Gunit}.\]
\end{enumerate}
\end{proposition}

\begin{proof}
$(i)\Rightarrow (ii)$: This direction follows the same lines as Proposition~\ref{min1}
except we choose compact open (as oppose to merely open) bisections every time.

$(ii)\Rightarrow (i)$:  If  $u\in \Gunit$ and $\overline{r(G_u)}\neq\Gunit$, we choose
$U\subseteq\Gunit\setminus\overline{r(G_u)}$, a non-empty close and open subset (this is
possible because these form a base for the topology of $\Gunit$. The rest of the proof is
identical to that of Proposition~\ref{min1}.

$(ii)\Rightarrow (iii)$: If $f\in C(\Gunit,\mathbb{Z})^+$ is non-zero, we may write
$f=\sum_{k=1}^{l}m_k\ch_{U_k}$ where the $m_k$ are strictly positive integers and
$U_k\subseteq\Gunit$ are closed and open and non-empty. Setting $U=U_1$, there are compact
and open bisections $E_1,\dots, E_n$ with
\[\bigcup_{j=1}^{n}U_{E_j}=\Gunit.\]
Note that
\[E_jf=E_j\left(\sum_{k=1}^{l}m_k\ch_{U_k}\right)=\sum_{k=1}^{l}m_k\ch_{r(E_jU_k)}\geq  m_1\ch_{r(E_jU_1)}\geq \ch_{r(E_jU)}.\]
Therefore
\[\ch_{\Gunit}=\ch_{\cup_{j=1}^{n}U_{E_j}}\leq \sum_{j=1}^{n}\ch_{U_{E_j}}=\sum_{j=1}^{n}\ch_{r(E_jU)}\leq \sum_{j=1}^{n}E_jf.\]
$(iii)\Rightarrow (ii)$: Let $\emptyset\neq U\subseteq\Gunit$ be closed and open, and set
$f=\ch_{U}$. By our assumption there are bisections $E_1,\dots, E_n\in\C$ with
$\sum_{j=1}^{n}E_jf\geq\ch_{\Gunit}$. Then
\[\ch_{\Gunit}\leq \sum_{j=1}^{n}E_jf=\sum_{j=1}^{n}E_j\ch_{U}=\sum_{j=1}^{n}\ch_{r(E_jU)}=\sum_{j=1}^{n}\ch_{U_{E_j}},\]
which shows that $\Gunit\subseteq \cup_{j=1}^{n}U_{E_j}$.
\end{proof}

Proposition~\ref{min2} part (iii) will be instrumental in constructing faithful traces on
completely non-paradoxical groupoid \cstar-algebras in the following section.

\section{The type semigroup of an ample groupoid}\label{sec:SG}

Given a transformation group $(X,\Gamma)$ where $\Gamma$ is discrete and $X$ is the Cantor set, R{\o}rdam and Sierakowski~\cite{RS} construct a type semigroup $S(X,\Gamma)$ that witnesses the pure infiniteness of the reduced \cstar-crossed product $C(X)\rtimes_{\lambda}\Gamma$. This construction is much in the spirit of the classical type semigroup of an arbitrary action studied in~\cite{Wa}. Given a (not necessarily commutative) \cstar-dynamical system $(A,\Gamma, \alpha)$ with $A$ stably finite and admitting refinement properties, the first author constructs in~\cite{Ra2} a type semigroup $S(A,\Gamma)$ that distinguishes stable finitness versus pure infiniteness of the reduced crossed product $A\rtimes_{\lambda}\Gamma$. In this section we construct the analogue in the setting of ample groupoids. We will restrict our attention to \'{e}tale, ample groupoids so that the unit space $\Gunit$ is totally disconnected. It then follows
that the set of compactly-supported, integer-valued, continuous functions on the unit space $C_c(\Gunit,\mathbb{Z})$ is a dimension group with positive cone $C_c(\Gunit,\mathbb{Z})^+$. We will show in Section~\ref{sec:reconcile} that our type semigroup $S(G)$ for a transformation groupoid coincides with R{\o}rdam and Sierakowski's semigroup $S(X, \Gamma)$.

\begin{definition}\label{defnrelation}
Let $G$  be an \'{e}tale and ample groupoid. Define a relation on
$C_c(\Gunit,\mathbb{Z})^+$ as follows: for $f,g\in C_c(\Gunit,\mathbb{Z})^+$, set $f\sim_{G}
g$ if there are compact open bisections $E_1,\dots, E_n\in\C$  with
\begin{equation}\label{relation}
f=\sum_{i=1}^{n}\mathds{1}_{s(E_i)},\quad \text{and}\quad g=\sum_{i=1}^{n}\mathds{1}_{r(E_i)}.
\end{equation}
\end{definition}

\begin{proposition}
Let $G$ be an \'{e}tale, ample groupoid. Then the relation on $C_c(\Gunit,\mathbb{Z})^+$
described in Definition~\ref{defnrelation} is an equivalence relation.
\end{proposition}

\begin{proof}
If $f\in C_c(\Gunit,\mathbb{Z})^+$, one can write $f=\sum_{j}\ch_{C_j}$ where the
$C_j\subseteq G^{(0)}$ are compact and open. Note that for each $j$, $C_j$ is actually
compact open bisection with $s(C_j)=C_j$ and $r(C_j)=C_j$. It easily follows that
$f\sim_{G}f$.

Let $f\sim g$ as in~\ref{relation}. For each $i$ set $F_i=E_i^{-1}$, which also belong to
$\C$. Then $r(F_i)=r(E_i^{-1})=s(E_i)$, and $s(F_i)=s(E_i^{-1})=r(E_i)$, whence $g\sim f$
since
\[g=\sum_{i=1}^{n}\mathds{1}_{r(E_i)}=\sum_{i=1}^{n}\mathds{1}_{s(F_i)},\quad \text{and}\quad f=\sum_{i=1}^{n}\mathds{1}_{s(E_i)}=\sum_{i=1}^{n}\mathds{1}_{r(F_i)}.\]

To prove transitivity, we suppose that $f\sim g\sim h$ via
\begin{equation}\label{transitive} f=\sum_{i=1}^{m}\ch_{s(E_i)},\quad g=\sum_{i=1}^{m}\ch_{r(E_i)},\quad g=\sum_{j=1}^{n}\ch_{s(F_j)},\quad h=\sum_{j=1}^{n}\ch_{r(F_j)},
\end{equation}
where the $E_i$ and $F_j$ are compact open bisections.

\begin{claim}
In the equations~\ref{transitive}, we may assume that $m=n$ and the bisections satisfy $r(E_i)=s(F_i)$ for $i=1,\dots,m$.
\end{claim}

By applying the refinement property to the two expressions for $g$ we can find compact
open subsets $\{Z_{i,j}\}_{i,j}$ of $\Gunit$ satisfying
\[\ch_{r(E_i)}=\sum_{j=1}^{n}\ch_{Z_{i,j}},\quad \ch_{s(F_j)}=\sum_{i=1}^{m}\ch_{Z_{i,j}}\]
for every $i$ and $j$. Note that $\bigsqcup_{j=1}^{n}Z_{i,j}=r(E_i)$ and $\bigsqcup_{i=1}^{m}Z_{i,j}=s(F_j)$. For each pair $(i,j)$ consider the compact open bisections
\[G_{i,j}:=E_i\cap r^{-1}(Z_{i,j})\subseteq E_i,\quad\text{and}\quad H_{i,j}:=F_j\cap s^{-1}(Z_{i,j})\subseteq F_j.\]
We then have $E_i=\bigsqcup_{j=1}^{n}G_{i,j}$ and $F_j=\bigsqcup_{i=1}^{m}H_{i,j}$. Thus
\begin{align*}f&=\sum_{i=1}^{m}\ch_{s(E_i)}=\sum_{i,j}\ch_{s(G_{i,j})},\quad g=\sum_{i=1}^{m}\ch_{r(E_i)}=\sum_{i,j}\ch_{r(G_{i,j})},\\
g&=\sum_{j=1}^{n}\ch_{s(F_j)}=\sum_{i,j}\ch_{s(H_{i,j})},\quad h=\sum_{j=1}^{n}\ch_{r(F_j)}=\sum_{i,j}\ch_{r(H_{i,j})}.
\end{align*}
Note that $r(G_{i,j})=Z_{i,j}=s(H_{i,j})$. This proves the Claim.

Now look at the compact open bisections $B_i=F_iE_i$, for $1\leq i\leq m$. Note that $s(B_i)=s(E_i)$ and $r(B_i)=r(F_i)$. We end up with $f\sim h$ since
\[f=\sum_{i=1}^{m}\ch_{s(E_i)}=\sum_{i=1}^{m}\ch_{s(B_i)},\quad h=\sum_{i=1}^{m}\ch_{r(F_i)}=\sum_{i=1}^{m}\ch_{r(B_i)},\]
and this finishes the proof.

\end{proof}

We can now make the following definition.

\begin{definition}\label{typesemigroup}
Let $G$ be an  \'{e}tale and ample groupoid. We define the \emph{type semigroup of $G$}
as  $S(G):=C_c(\Gunit,\mathbb{Z})^{+}/\sim_{G}$, and write $[f]_{G}$ for the equivalence
class with representative $f\in C_c(\Gunit,\mathbb{Z})^{+}$.
\end{definition}

\begin{remark}\label{rmk:BL matchup}
Our definition of the type semigroup of $G$ is related to the definition given by B\"onicke and Li in \cite[Definition~5.3]{BL}, though our definition is somewhat more algebraic and does not involve amplifying and
passing to levels of $\Gunit \times \mathbb{N}$. The apparent discrepancy can be
resolved using Proposition~\ref{prp:stable equiv} below: Write $\mathcal{R} := \mathbb{N}
\times \mathbb{N}$ regarded as a discrete principal groupoid, and let $\mathcal{K}G$
denote the groupoid $G \times \mathcal{R}$. Identifying $\mathcal{R}^{(0)}$ with
$\mathbb{N}$ in the obvious way, it is straightforward to see that the type semigroup of
$G$ as defined in \cite[Definition~5.3]{BL} is precisely the type semigroup of
$\mathcal{K}G$ as defined by Definition~\ref{typesemigroup}. Proposition~\ref{prp:stable
equiv} shows that the type semigroups, in the sense of our definition, of $\mathcal{K}G$
and of $G$ coincide, so we see that our definition of the type semigroup of $G$ agrees
with \cite[Definition~5.3]{BL} up to canonical isomorphism.
\end{remark}

We can define addition of classes simply by $[f]_G+[g]_\alpha:=[f+g]_G$ for $f,g$ in
$C_c(\Gunit,\mathbb{Z})^{+}$. It is routine to check that this operation is well defined;
indeed if $f\sim_{G} f'$ and $g\sim_{G} g'$ via
\[f=\sum_{i=1}^{n}\ch_{s(E_i)},\quad f'=\sum_{i=1}^{n}\ch_{r(E_i)},\quad g=\sum_{j=1}^{m}\ch_{s(F_j)},\quad g'=\sum_{j=1}^{m}\ch_{r(F_j)}, \]
then
\begin{align*} [f]_G+[g]_G &=[f+g]_G=\bigg[\sum_{i=1}^{n}\ch_{s(E_i)}+\sum_{j=1}^{m}\ch_{s(F_j)}\bigg]_{G}
=\bigg[\sum_{i=1}^{n}\ch_{r(E_i)}+\sum_{j=1}^{m}\ch_{r(F_j)}\bigg]_{G}\\
&=[f'+g']_{G}=[f']_{G}+[g']_{G}.
\end{align*}
We make a few elementary observations about $S(G)$ for such an ample, \'{e}tale $G$.
Firstly, $S(G)$  is not only a semigroup but an abelian monoid as $[0]_{G}$ is clearly
the neutral additive element. Impose the algebraic ordering on $S(G)$, that is, set
$[f]_{G}\leq [g]_{G}$ if there is an $h\in C_c(\Gunit,\mathbb{Z})^+$ with
$[f]_{G}+[h]_{G}=[g]_{G}$. This gives $S(G)$ the structure of a pre-ordered abelian
monoid. Note that if $f,g\in C_c(\Gunit,\mathbb{Z})^{+}$ with $f\leq g$ (in the ordering
of $C_c(\Gunit,\mathbb{Z})$) then $[f]_G\leq[g]_{G}$ in $S(G)$. To see this, $f\leq g$
implies $g-f:=h\in C_c(\Gunit,\mathbb{Z})^+$, so $[g]_G=[f+h]_G=[f]_G+[h]_\alpha$ which
gives $[f]_G\leq[g]_G$. Next, we observe that if $[f]_G=[0]_G$, for some $f$ in
$C_c(\Gunit,\mathbb{Z})^{+}$, then in fact $f=0$ in $C_c(\Gunit,\mathbb{Z})$. Indeed, if
$f=\sum_{i}\ch_{s(E_i)}$, and $\sum_{i}\ch_{r(E_i)}=0$ for some bisections $E_1,\dots,
E_n\in\C$ , then $r(E_i)=\emptyset$ for every $i$ which implies $E_i=\emptyset$ and
$s(E_i)=\emptyset$ for all $i$. All together, there is a surjective, order preserving,
faithful, monoid homomorphism
\[\pi: C_c(\Gunit,\mathbb{Z})^+\longrightarrow S(G)\quad\mbox{given by}\quad \pi(f)=[f]_G.\]

In the remainder of this section we show that the type semigroup of an ample groupoid is an invariant for groupoid equivalence in the sense that equivalent groupoids have isomorphic type semigroups. This will also enable us to extend our characterization of stable finiteness for reduced $\cstar$-algebras of minimal ample groupoids from the situation of groupoids
with compact unit space to the general case.

We first show that the type semigroup is an isomorphism invariant.

\begin{proposition}\label{isomorphisminvariant}
Let $\varphi:G\rightarrow H$ be an isomorphism of \'{e}tale and ample groupoids. The
induced map $\overline{\varphi}: C_c(\Hunit,\mathbb{Z})^+\longrightarrow
C_c(\Gunit,\mathbb{Z})^+$ given by $\overline{\varphi}(f)=f\circ\varphi|_{\Gunit}$ drops
to an isomorphism of monoids $\phi: S(H)\rightarrow S(G)$.
\end{proposition}
\begin{proof}
Since $\varphi$ is a homeomorphism, its restriction
$\varphi|_{\Gunit}:\Gunit\rightarrow\Hunit$ is a homeomorphism, so $\overline{\varphi}$
is a well-defined monoid isomorphism. Moreover, if $A\subseteq \Hunit$ is a compact open subset, then $\overline{\varphi}(\ch_A)=\ch_{\phi^{-1}(A)}$.

If $f\sim_{H}g$ in $C_c(\Hunit,\mathbb{Z})^+$ with compact open bisections $E_1,\dots, E_n$ in $H$ satisfying
\[
    f=\sum_{i=1}^{n}\ch_{s(E_i)},\quad\text{and}\quad g=\sum_{i=1}^{n}\ch_{r(E_i)},
\]
then $\overline{\varphi}(f)\sim_{G}\overline{\varphi}(g)$ since
\[
\overline{\varphi}(f)
    =\sum_{i=1}^{n}\ch_{\psi(s(E_i))}
    =\sum_{i=1}^{n}\ch_{s(\psi(E_i))},
        \quad\text{and}\quad
\overline{\varphi}(g)
    =\sum_{i=1}^{n}\ch_{\psi(r(E_i))}
    =\sum_{i=1}^{n}\ch_{r(\psi(E_i))}.
\]

Hence there is a homomorphism $\phi: S(H)\rightarrow S(G)$ such that
$\phi([f]_H)=[\overline{\varphi}(f)]_G$. The same argument applied to $\varphi^{-1} : H \to G$ yields an inverse to $\phi$, so $\phi$ is an isomorphism.
\end{proof}

To see that equivalent ample groupoids have isomorphic type semigroups, we will appeal to the groupoid analogue \cite{CRS} of the Brown--Green--Rieffel theorem \cite{BGR}. We first need to know that the type semigroup is invariant under stabilization of groupoids.

In the following, we write $\mathcal{R}$ for the discrete equivalence relation $\mathcal{R}=\mathbb{N} \times \mathbb{N}$ regarded as a principal ample groupoid with unit space $\mathcal{R}^{(0)}=\{(n,n)\ |\ n\in\mathbb{N}\}$ identified with $\mathbb{N}$. The stabilization of a groupoid $G$ is the product $\mathcal{K}G:=G\times\mathcal{R}$ with its natural groupoid structure. Note that $\mathcal{K}\Gunit=\Gunit\times\mathcal{R}^{(0)}$. If $G$ is ample, one notes that $E\times\{(i,j)\}\subseteq\mathcal{K}G$ is a compact open bisection if and only if $E\subseteq G$ is a compact open bisection. Moreover, a brief compactness argument shows that every compact open bisection $F\subseteq\mathcal{K}G$ can be written as a disjoint union $F=\bigsqcup_{k=1}^{K}E_k\times \{(i_k, j_k)\}$ where $E_k$ are compact open bisections in $G$ and $(i_k,j_k)\in\mathcal{R}$. Likewise,  any compact open $A\subseteq\mathcal{K}\Gunit$ can be written as a disjoint union $A=\bigsqcup_{k=1}^{K}A_k\times \{(i_k, i_k)\}$ where the  $A_k$ are compact open subsets of $\Gunit$ and $i_k\in\mathbb{N}$.

We will write $\ch_{(m,n)}$ for the point-mass functions in $C_c(\mathcal{R})$. Given functions $f\in C_c(G)$ and $g\in C_c(\mathcal{R})$, we denote by $f\times g :G\times \mathcal{R}\to\mathbb{C}$ the function defined by sending $(\alpha,(m,n))\mapsto f(\alpha)g((m,n))$, where $\alpha\in G$, $(m,n)\in\mathcal{R}$. Note that the association $(f,g)\mapsto f\times g$ is a well-defined bilinear map $C_c(G)\times C_c(\mathcal{R})\to C_c(\mathcal{K}G)$. Also, if $A\subseteq G$ and $B\subseteq\mathcal{R}$ are compact open, then $\ch_{A}\times\ch_B=\ch_{A\times B}$.

\begin{proposition}\label{prp:stable equiv}
Let $G$ be an ample Hausdorff groupoid.
\begin{enumerate}
\item [(i)] For any $n\in\mathbb{N}$ and $f \in C_c(G^{(0)}, \mathbb{Z})^+$, we have
\[
[f \times \ch_{(0,0)}]_{\mathcal{K}G} = [f \times \ch_{(n,n)}]_{\mathcal{K}G}\quad\text{in}\quad S(\mathcal{K}G).
\]
\item[(ii)] There is an isomorphism of monoids $\varphi: S(G) \to
S(\mathcal{K}G)$ satisfying
\[\varphi([f]_G) = [f \times \ch_{(0,0)}]_{\mathcal{K}G},\quad f \in
C_c(G^{(0)}, \mathbb{Z})^+.\]
\end{enumerate}
\end{proposition}

\begin{proof}
Let $n\in\mathbb{N}$ and $f\in C_c(\Gunit,\mathbb{Z})^{+}$. We can write $f$ as a finite sum $f=\sum_{k}m_k\ch_{A_k}$ with $m_k\in\mathbb{Z}^{+}$ and $A_k\subseteq\Gunit$ are compact open subsets. We then see that
\begin{align*}
f\times\ch_{(0,0)}&=\bigg(\sum_{k}m_k\ch_{A_k}\bigg)\times\ch_{(0,0)}=\sum_{k}m_k(\ch_{A_k}\times\ch_{(0,0)})=\sum_{k}m_k\ch_{A_k\times\{(0,0)\}}\\&=\sum_{k}m_k\ch_{s(A_k\times\{(n,0)\})}\sim\sum_{k}m_k\ch_{r(A_k\times\{(n,0)\})}=\sum_{k}m_k\ch_{A_k\times\{(n,n)\}}=f\times\ch_{(n,n)},
\end{align*}
which proves (i).

The map $C_c(\Gunit,\mathbb{Z})^+\to S(\mathcal{K}G)$ which sends
$f\mapsto [f\times\ch_{(0,0)}]_{\mathcal{K}G}$ is clearly well-defined and additive. If $f \sim_{G} g$ in $C_c(G, \mathbb{Z})^+$, then there exist compact open bisections $E_1,
\dots, E_n$ in $G$ such that $f = \sum_{i=1}^{n} \ch_{s(E_i)}$ and $g = \sum \ch_{r(E_i)}$. Hence
\begin{align*}
f \times \ch_{(0,0)}& =  \bigg(\sum_{i=1}^{n} \ch_{s(E_i)}\bigg)\times \ch_{(0,0)}= \sum_{i=1}^{n} \ch_{s(E_i)}\times\ch_{(0,0)}= \sum_{i=1}^{n} \ch_{s(E_i)\times\{(0,0)\}}\\ &=\sum_{i=1}^{n} \ch_{s(E_i\times\{(0,0)\})}\sim\sum_{i=1}^{n} \ch_{r(E_i\times\{(0,0)\})}=g \times \ch_{(0,0)}.
\end{align*}
There is, therefore, a well-defined monoid homomorphism
$\varphi : S(G) \to S(\mathcal{K}G)$ satisfying the description in (ii).

To see that $\varphi$ is surjective, fix $h \in C_c(\mathcal{K}\Gunit,\mathbb{Z})^+$ and write $h = \sum_{k} m_k\ch_{A_k}$ for a finite list of positive integers $m_k$ and compact open subsets $A_k\subseteq\mathcal{K}\Gunit$.  By our discussion before the statement of the Proposition we may assume that each $A_k=B_k\times\{(i_k, i_k)\}$ with $B_k\subseteq\Gunit$ compact and open. Setting $f=\sum_km_k\ch_{B_k}$ in $C_c(\Gunit,\mathbb{Z})^{+}$ we get
\begin{align*}
[h]_{\mathcal{K}G}& = \bigg[\sum_{k} m_k\ch_{B_k\times\{(i_k, i_k)\}}\bigg]=\sum_{k} m_k[\ch_{B_k}\times\ch_{(i_k, i_k)}]\stackrel{(i)}{=}\sum_{k} m_k[\ch_{B_k}\times\ch_{(0, 0)}]\\&=\sum_{k} m_k\varphi([\ch_{B_k}])=\varphi\bigg(\sum_{k} m_k[\ch_{B_k}]\bigg)=\varphi([f]_G).
\end{align*}

We show that $\varphi$ is injective by constructing a left inverse. The map $\rho:C_c(\mathcal{K}\Gunit,\mathbb{Z})^+\to C_c(\Gunit,\mathbb{Z})^+$ given by
\[
\rho(h)(u) := \sum_{n \in \mathbb{N}} h(u,n),\quad u\in\Gunit
\]
is clearly well-defined and additive. Now if $F=E\times\{(i,j)\}$ is a compact open bisection in $\mathcal{K}G$, we see that for all $u\in\Gunit$
\begin{align*}
\rho(\ch_{s(F)})(u)
    &=\sum_n\ch_{s(F)}(u,n) =\sum_n\ch_{s(E)\times\{(j,j)\}}(u,n)\\
    &=\sum_n(\ch_{s(E)}\times\ch_{(j,j)})(u,n)=\ch_{s(E)}(u),
\end{align*}
whence $\rho(\ch_{s(F)})=\ch_{s(E)}$. Similarly, $\rho(\ch_{r(F)})=\ch_{r(E)}$.  Given
any compact open bisection $F\subseteq\mathcal{K}G$,  we can write
$F=\bigsqcup_{k=1}^{K}F_k$ with $F_k=E_k\times \{(i_k, j_k)\}$, and we get
\begin{align*}
\rho(\ch_{s(F)})&=\rho\bigg(\sum_k\ch_{s(F_k)}\bigg)=\sum_k\rho(\ch_{s(F_k)})=\sum_k\ch_{s(E_k)}\sim \sum_k\ch_{r(E_k)}\\&=\sum_k\rho(\ch_{r(F_k)})=\rho\bigg(\sum_k\ch_{r(F_k)}\bigg)=\rho(\ch_{r(F)}).
\end{align*}
Since $\rho$ is additive  we conclude that if $h \sim g$ in $C_c(\mathcal{K}\Gunit,\mathbb{Z})^+$ then $\rho(h) \sim \rho(g)$ in $C_c(G^{(0)}, \mathbb{Z})^+$. Thus $\rho$ descends to a monoid homomorphism $\tilde{\rho}: S(\mathcal{K}G) \to S(G)$ given by $\tilde{\rho}([h]_{\mathcal{K}G})=[\rho(h)]_G$. It is quite clear that $\rho(f\times\ch_{(0,0)})=f$, so $\tilde{\rho}$ is the desired inverse of $\varphi$.
\end{proof}

\begin{corollary}
Let $G$ and $H$ be \'{e}tale, ample groupoids with $\sigma$-compact unit spaces. If $G$ and $H$ are groupoid equivalent, then $S(G) \cong S(H)$.
\end{corollary}
\begin{proof}
We know that $S(G) \cong S(\mathcal{K}G)$ and $S(H) \cong S(\mathcal{K}H)$ by Proposition~\ref{prp:stable equiv}. Since $G$ and $H$ are equivalent, \cite[Theorem~2.1]{CRS} shows that $\mathcal{K}G\cong \mathcal{K}H$, so the result follows from Proposition~\ref{isomorphisminvariant}.
\end{proof}

\section{Stably finite groupoid \texorpdfstring{\cstar}{C*}-algebras}\label{sec:stabfin}

This section is concerned with characterizing stably finite \cstar-algebras that arise
from ample groupoids by the non-paradoxical nature of its type semigroup and by a
K-theoretic coboundary property analogous to that seen in the work of N.\
Brown~\cite{BrownAFE}. Our unified statements (Theorem~\ref{theoremsfcnp} and
Corollary~\ref{corollarysfcnp}) are the main results of this section and include, as
promised, a converse to Proposition~\ref{sfimpliescnp}.

We first look at how the type semigroup $S(G)$ of an ample groupoid $G$ reflects the
notion of paradoxicality introduced in Definition~\ref{defnparadox}.

\begin{lemma}\label{lemmaparadoxtype}
Let $G$ be an \'{e}tale and ample groupoid, and let $A\subseteq \Gunit$ be a non-empty
compact open subset. Writing $\theta=[\ch_A]_G$ in $S(G)$,  $A$ is $(k,l)$-paradoxical if
and only if $k\theta\leq l\theta$ in $S(G)$.
\end{lemma}

\begin{proof}
Suppose $A$ is $(k,l)$-paradoxical and that $E_1,\dots, E_n$ are bisections in $\C$
satisfying~\ref{relation}. Then
\[k\theta=k[\ch_A]_G=[k\ch_A]_G\leq\left[\sum_{i=1}^{n}\ch_{s(E_i)}\right]_G=\left[\sum_{i=1}^{n}\ch_{r(E_i)}\right]_G\leq[l\ch_A]_G=l[\ch_A]_G=l\theta.\]

Conversely, suppose $k\theta\leq l\theta$. Then there is an $f\in C(\Gunit,\mathbb{Z})^+$
such that
\[ [k\ch_A+f]_G=[k\ch_A]_G+[f]_G=k[\ch_A]_G+[f]_G=l[\ch_A]_G=[l\ch_A]_G.\]
This means that there are compact open bisections $E_1,\dots, E_n$ with
\[k\ch_A+f=\sum_{i=1}^{n}\ch_{s(E_i)},\quad\text{and}\quad \sum_{i=1}^{n}\ch_{r(E_i)}=l\ch_A.\]
It follows that $A$ is $(k,l)$-paradoxical since $k\ch_A\leq k\ch_A+f$  and
so~\ref{klparadox} is satisfied.
\end{proof}

Before moving forward we recall some terminology for pre-ordered abelian monoids $(S,+)$.
For positive integers $k>l>0$, we say that an element $\theta\in S$ is
\emph{$(k,l)$-paradoxical} provided that $k\theta\leq l\theta$. If $\theta$ fails to be
$(k,l)$-paradoxical for all pairs of integers $k>l>0$, call $\theta$ \emph{completely
non-paradoxical}. Note that $\theta$ is completely non-paradoxical if and only if
$(n+1)\theta\nleq n\theta$ for all $n\in\mathbb{N}$. If every element in $S$ is
completely non-paradoxical we say that $S$ is completely non-paradoxical. The above lemma
basically states that in its setting, an subset $A\subseteq G^{(0)}$ is completely
non-paradoxical precisely when $[\ch_A]_G$  is completely non-paradoxical in the
pre-ordered abelian monoid $S(G)$, and $G$ is completely non-paradoxical if and only if
$S(G)$ is completely non-paradoxical. A \emph{state} on $S$ is a map
$\nu:S\rightarrow[0,\infty]$ which is additive, respects the pre-ordering $\leq$, and
satisfies $\nu(0)=0$. If a state assumes a value other than $0$ or $\infty$,  it said to
be \emph{non-trivial}.

The following result is a main ingredient in the proof of Tarski's theorem. It is a
Hahn-Banach type extension result and is essential in establishing a converse to
Proposition~\ref{sfimpliescnp}. A proof can be found in~\cite{Wa}.

\begin{theorem}\label{Tarski} Let $(S,+)$ be an abelian monoid equipped with the algebraic ordering, and let $\theta$ be an element of $S$. Then the following are equivalent:
\begin{enumerate}
\item[(i)] $(n+1)\theta\nleq n\theta$ for all $n\in\mathbb{N}$, that is $\theta$ is
    completely non-paradoxical.
\item[(ii)] There is a non-trivial state $\nu: S\rightarrow[0,\infty]$ with
    $\nu(\theta)=1$.
\end{enumerate}
\end{theorem}

For the next lemma we need the notion of an invariant state. If $G$ is an \'{e}tale,
ample groupoid, we will call a state $\beta: C_c(\Gunit,\mathbb{Z})\rightarrow\mathbb{R}$
\emph{invariant} if $\beta(\ch_{s(E)})=\beta(\ch_{r(E)})$ for any compact open bisection
$E$.

\begin{lemma}\label{lemmacnp} Let $G$ be an \'{e}tale and ample groupoid with compact unit space $\Gunit$. Consider the following properties.
\begin{enumerate}
\item[(i)] For every non-empty closed and open $U\subseteq\Gunit$, there is a faithful invariant positive group homomorphism $\beta:C(\Gunit,\mathbb{Z})\rightarrow\mathbb{R}$ with $\beta(\ch_{U})=1$.
\item[(ii)] There is a faithful invariant state $\beta$ on the dimension group
    $$(C(\Gunit,\mathbb{Z}), C(\Gunit,\mathbb{Z})^+, \ch_{\Gunit}).$$
\item[(iii)] $G$ is completely non-paradoxical.
\end{enumerate}
The implications $(i)\Rightarrow(ii)\Rightarrow(iii)$ always hold. If $G$ is minimal then $(iii)\Rightarrow(i)$ whence all conditions are equivalent.
\end{lemma}

\begin{proof}
$(i)\Rightarrow(ii)$: Simply take $U=\Gunit$.

$(ii)\Rightarrow(iii)$: Suppose a compact open subset $A\subseteq\Gunit$ is
$(k,l)$-paradoxical for a pair of positive integers $k>l>0$. There are then $E_1,\dots,
E_n\in\C$ with
\[k\ch_A\leq\sum_{j=1}^{n}\ch_{s(E_j)},\quad\text{and}\quad \sum_{j=1}^{n}\ch_{r(E_j)}\leq l\ch_{A}.\]
Applying $\beta$ one gets
\begin{align*}
k\beta(\ch_{A})&=\beta(k\ch_A)\leq\beta\bigg(\sum_{j=1}^{n}\ch_{s(E_j)}\bigg)=\sum_{j=1}^{n}\beta(\ch_{s(E_j)})=\sum_{j=1}^{n}\beta(\ch_{r(E_j)})
=\beta\bigg(\sum_{j=1}^{n}\ch_{r(E_j)}\bigg)\\&\leq \beta(l\ch_A)=l\beta(\ch_{A}).
\end{align*}
If $A$ is non-empty then $\ch_{A}\neq0$ and since $\beta$ is faithful we may divide by
$\beta(\ch_{A})>0$ to yield $k\leq l$, which is contradictory. Therefore $A=\emptyset$
and $G$ is completely non-paradoxical.

Assuming $G$ is minimal we prove $(iii)\Rightarrow(i)$. Let
$U\subseteq\Gunit$ be a non-empty compact open set, and consider $\theta=[\ch_{U}]_G$ in
$S(G)$. Since $G$ is completely non-paradoxical $\theta$ is completely non-paradoxical
and therefore Theorem~\ref{Tarski} provides a non-trivial state $\nu: S(G)\rightarrow
[0,\infty]$ with $\nu(\theta)=1$. Composing this state with the quotient mapping $\pi:
C(\Gunit,\mathbb{Z})^+\rightarrow S(G)$ yields
\[\beta:=\nu\pi: C(\Gunit,\mathbb{Z})^+\longrightarrow[0,\infty],\]
an additive, order-preserving map with $\beta(0)=0$ and $\beta(\ch_{U})=1$. By
construction $\beta$ is invariant; indeed if $E$ is a compact open bisection in $G$ then
\[\beta(\ch_{s(E)})=\nu\pi(\ch_{s(E)})=\nu([\ch_{s(E)}]_G)=\nu([\ch_{r(E)}]_G)=\nu\pi(\ch_{r(E)})=\beta(\ch_{r(E)}).\]

We claim that $\beta$ is in fact finite. Since $G$ is minimal, there are compact open
bisections $E_1,\dots, E_n$ with
\[\ch_{\Gunit}\leq\sum_{j=1}^{n}E_j\ch_{U}=\sum_{j=1}^{n}\ch_{r(E_jU)}.\]
Note that for a bisection $E\in\C$ we have $s(EU)\subseteq U$ and so $\ch_{s(EU)}\leq
\ch_{U}$. Using this fact and applying $\beta$ to the above inequality gives
\[\beta(\ch_{\Gunit})\leq\beta\bigg(\sum_{j=1}^{n}\ch_{r(E_jU)}\bigg)=\sum_{j=1}^{n}\beta(\ch_{r(E_jU)})=\sum_{j=1}^{n}\beta(\ch_{s(E_jU)})\leq \sum_{j=1}^{n}\beta(\ch_{U})=n<\infty,\]
where we have used the invariance and order-preserving properties of $\beta$. If $f\in
C(\Gunit,\mathbb{Z})^+$, write $f=\sum_{k=1}^{K}m_k\ch_{A_k}$, where $m_k\in\mathbb{Z}^+$
and the $A_k$ are closed and open subsets of \Gunit. Since $\ch_{A_k}\leq\ch_{\Gunit}$ we
have
\[\beta(f)=\beta\bigg(\sum_{k=1}^{K}m_k\ch_{A_k}\bigg)=\sum_{k=1}^{K}m_k\beta(\ch_{A_k})\leq \beta(\ch_{\Gunit})\sum_{k=1}^{K}m_k\]
which is finite.

Now we show that $\beta$ is faithful. Suppose that $A\subseteq\Gunit$ is a non-empty closed
and open subset with $\beta(\ch_A)=0$. Again we use minimality to find compact open
bisections $E_1,\dots, E_n$ satisfying
\[\ch_{\Gunit}\leq\sum_{j=1}^{n}E_j\ch_{A}=\sum_{j=1}^{n}\ch_{r(E_jA)}.\]
Applying $\beta$ and using its properties we get
\begin{align*}1=\beta(\ch_U)\leq \beta(\ch_{\Gunit})&\leq \beta\bigg(\sum_{j=1}^{n}\ch_{r(E_jA)}\bigg)=\sum_{j=1}^{n}\beta(\ch_{r(E_jA)})=\sum_{j=1}^{n}\beta(\ch_{s(E_jA)})\\&\leq \sum_{j=1}^{n}\beta(\ch_A)=0
\end{align*}
which is absurd. Therefore $\beta(\ch_A)>0$ for every non-empty closed and open $A\subseteq\Gunit$.
Now suppose $f\in C(\Gunit,\mathbb{Z})^+$ is non-zero. We may write
$f=\sum_{k=1}^{K}m_k\ch_{A_k}$ where each $A_k\subseteq\Gunit$ is non-empty, closed and open, and
every $m_k$ is a strictly positive integer. Since $\beta(\ch_{A_k})>0$ for every $k$, it easily follows that $\beta(f)>0$ as well.

To complete the proof we need only extend $\beta$ to all of $C(\Gunit,\mathbb{Z})$; but
this is easily done by expressing $f\in C(\Gunit,\mathbb{Z})$ as the difference of its
positive and negative parts $f=f^+-f^-$, $f^+,f^-\in C(\Gunit,\mathbb{Z})^+$ and applying
$\beta$ additively.
\end{proof}

With lemma~\ref{lemmacnp} at our disposal we can characterize stably finite \cstar-algebras that arise from \'{e}tale ample groupoids.  Before we delve into the details, it is natural to tie stable finiteness with the idea of a coboundary subgroup.

\begin{definition} Let $G$ be an \'{e}tale, ample groupoid and let $\C$ denote the family of all compact open bisections in $G$. We define the \emph{coboundary subgroup of $G$} as the subgroup of $C_c(\Gunit,\mathbb{Z})$ generated by differences $\ch_{s(E)}-\ch_{r(E)}$, $E\in\C$, that is
\[H_G:=\big\langle \ch_{s(E)}-\ch_{r(E)}\ |\ E\in\C\big\rangle.\]
We say that $G$ satisfies the \emph{coboundary condition} if $H_G\cap
C_c(\Gunit,\mathbb{Z})^+=\{0\}$.
\end{definition}

Here is the main result of this section.

\begin{theorem}\label{theoremsfcnp}
Let $G$ be an \'{e}tale groupoid with compact unit space \Gunit.
Consider the following properties.
\begin{enumerate}
\item[(i)] The \cstar-algebra $\cstar_{r}(G)$ admits a faithful tracial state.
\item[(ii)] The \cstar-algebra $\cstar_{r}(G)$ is stably finite.
\item[(iii)] $G$ satisfies the coboundary condition.
\item[(iv)] $G$ is completely non-paradoxical.
\end{enumerate}

The implications $(i)\Rightarrow(ii)\Rightarrow(iii)\Rightarrow(iv)$ always hold. If $G$
is minimal, then $(iv)\Rightarrow(i)$ and all properties are equivalent.

If $G$ is minimal and amenable, then properties $(i)$ through $(iv)$ are all equivalent to
\begin{enumerate}
\item[(v)] The \cstar-algebra $\cstar_{r}(G)$ is quasidiagonal.
\end{enumerate}
\end{theorem}

\begin{proof}
$(i)\Rightarrow(ii)$: Any unital \cstar-algebra that admits a faithful tracial state is stably finite.

$(ii)\Rightarrow(iii)$: Applying the \Ko-functor to the canonical inclusion $\iota:
C(\Gunit)\hookrightarrow\cstar_r(G)$ gives a positive group homomorphism
\[\Ko(\iota):C(\Gunit,\mathbb{Z})\cong\Ko(C(\Gunit))\longrightarrow\Ko(\cstar_r(G)).\]
The fact that $\cstar_r(G)$ is stably finite ensures that $\Ko(\iota)$ is faithful (see
section~\ref{Kth&traces}), that is, its kernel contains no non-zero positive elements.
However, we know that $\Ko(\iota)(\ch_s(E))=\Ko(\iota)(\ch_r(E))$ for every compact open
bisection $E$ (once more see section~\ref{Kth&traces}), therefore
$H_G\subseteq\ker(\Ko(\iota))$. It now follows that $H_G\cap C(\Gunit,\mathbb{Z})^+=\{0\}$.

$(iii)\Rightarrow(iv)$: Assume, by way of contradiction, that $G$ is not completely
non-paradoxical so that we can find a non-empty closed and open $A\subseteq\Gunit$,
positive integers $k>l>0$ and compact open bisections $E_1,\dots, E_n\in\C$ satisfying
\[k\ch_{A}\leq\sum_{i=1}^{n}\ch_{s(E_i)},\quad\text{and}\quad \sum_{i=1}^{n}\ch_{r(E_i)}\leq l\ch_{A}.\]
Then
\[0<(k-l)\ch_A=k\ch_A-l\ch_A\leq\sum_{i=1}^{n}\ch_{s(E_i)}-\sum_{i=1}^{n}\ch_{r(E_i)}=\sum_{i=1}^{n}\left(\ch_{s(E_i)}-\ch_{r(E_i)}\right)\]
which certainly belongs to $H_G\cap C(\Gunit,\mathbb{Z})^+$, a contradiction. Thus $G$ is
completely non-paradoxical.

Assuming $G$ is minimal and completely non-paradoxical we prove $(iv)\Rightarrow(i)$. By
Lemma~\ref{lemmacnp} there is an invariant faithful state $\beta:
C(\Gunit,\mathbb{Z})\rightarrow\mathbb{R}$. Composing with the isomorphism of dimension
groups $\dim:\Ko(C(\Gunit))\cong C(\Gunit,\mathbb{Z})$ gives us a state on the \Ko-group
\[\tilde{\beta}:=\beta\circ\dim:\Ko(C(\Gunit))\longrightarrow\mathbb{R}.\]
States on $\Ko(C(\Gunit))$ arise from traces, so let $\tau$ be the tracial state on
$C(\Gunit)$ such that $\Ko(\tau)=\tilde{\beta}$. Moreover, $\tau$ is simply integration
against a regular, Borel, probability measure $\mu$ on $\Gunit$. All together, if
$A\subseteq\Gunit$ is closed and open we get
\[\mu(A)=\int_{\Gunit}\ch_Ad\mu=\tau(\ch_A)=\Ko(\tau)([\ch_A]_0)=\tilde{\beta}([\ch_A]_0)=\beta\circ\dim([\ch_A]_0)=\beta(\ch_A).\]
The $G$-invariance of $\mu$ now follows from the invariance of $\beta$, indeed, if $E$ is
a compact open bisection, then
\[\mu(s(E))=\beta(\ch_{s(E)})=\beta(\ch_{r(E)})=\mu(r(E)).\]

We also claim that $\tau$ is faithful. To see this  we note that the measure $\mu$ has
full support. For if $\mu(U)=0$ for a non-empty open subset $U\subseteq\Gunit$, we can find
a non-empty closed and open $A\subseteq U$ with $\mu(A)=0$. It then follows that
$\beta(\ch_{A})=0$ which contradicts the fact that $\beta$ is faithful. Thus no such $U$
exists and so $\mu$ has full support whence $\tau$ is faithful. Finally, composing $\tau$
with the faithful conditional expectation $\mathbb{E}:\cstar_r(G)\rightarrow C(\Gunit)$
gives a faithful tracial state $\tau\circ\mathbb{E}:\cstar_r(G)\rightarrow\mathbb{C}$ as
desired.

For an amenable $G$, we have that $\cstar_r(G)$ is separable, nuclear, and satisfies the UCT. If, moreover, $\cstar_r(G)$ admits a faithful trace, then the main result in~\cite{TWW} ensures that $\cstar_r(G)$ is quasidiagonal.
\end{proof}

Using Proposition~\ref{prp:stable equiv}, we can extend the key statement of Theorem~\ref{theoremsfcnp} to the situation of groupoids with non-compact unit space.

\begin{corollary}\label{corollarysfcnp}
Let $G$ be an \'{e}tale, minimal, and ample groupoid. The following are
equivalent:
\begin{enumerate}
\item[(i)] $\cstar_r(G)$ admits a faithful semifinite trace $\tau$ such that $0 <\tau(\ch_K) < \infty$ for every compact open $K \subseteq G^{(0)}$;
\item[(ii)] $\cstar_r(G)$ is stably finite; and
\item[(iii)] $G$ satisfies the coboundary condition.
\end{enumerate}

If $G$ is also amenable, then properties $(i)$ through $(iii)$ are all equivalent to
\begin{enumerate}
\item[(iv)] The \cstar-algebra $\cstar_{r}(G)$ is quasidiagonal.
\end{enumerate}

\end{corollary}

\begin{proof}
If $\cstar_r(G)$ admits a nontrivial faithful semifinite trace as in~(i) then it is
stably finite because the collection $\{\ch_{K}\ |\ K \subseteq G^{(0)}\ \text{compact
open}\ \}$ forms an approximate identity for $\cstar_r(G)$.

Fix a compact open set $K \subseteq G^{(0)}$ and set $H := KGK = \{\gamma \in G :
r(\gamma), s(\gamma) \in K\}$. Since $G$ is minimal, we have $G^{(0)} = \{r(\gamma) :
s(\gamma) \in K\}$. Hence $G K$ is a $G$--$H$ equivalence. This gives $\mathcal{K}G \cong
\mathcal{K}H$ by \cite[Theorem~2.1]{CRS}. Also $\cstar_r(G)$ and $\cstar_r(H)$ are stably
isomorphic, and we deduce that $\cstar_r(G)$ is stably finite if and only if
$\cstar_r(H)$ is. By the proof of Proposition~\ref{prp:stable equiv} we see that $G$
satisfies the coboundary condition if and only its stabilization $\mathcal{K}G$ does.
Using with the fact that $G$ and $H$ are stably isomorphic we learn that $G$ satisfies
the coboundary condition if and only if $H$ does which in turn occurs if and only if
$\cstar_r(H)$ is stably finite by Theorem~\ref{theoremsfcnp}. Thus $\cstar_r(G)$ is
stably finite if and only if $G$ satisfies the coboundary condition.

Moreover, if $\cstar_r(G)$ is stably finite, then so is $\cstar_r(H)$, so there is a
faithful trace $\tau$ on $\cstar_r(H)$. From this we obtain a faithful semifinite trace
$\tilde\tau$ on $\cstar_r(\mathcal{K}H)$ satisfying $\tilde\tau(f) = \sum_{n \in
\mathbb{N}} \tau(f|_{H^{(0)} \times \{(n,n)\}})$ for $f \in C_c(\mathcal{K}H)$. Since
$\mathcal{K}G \cong \mathcal{K}H$, the map $\tilde\tau$ induces a semifinite trace on
$\cstar_r(\mathcal{K}G)$ which is nonzero and finite on the indicator function of any
compact open set of units. Restricting this semifinite trace to the corner $\cstar_r(G)$
of $\cstar_r(\mathcal{K}G)$ gives the result.

If $G$ is amenable and $\cstar_r(G)$ is stably finite, then $\cstar_r(H)$ is unital,
separable, nuclear, has a faithful trace and satisfies the UCT. Qasidiagonality of
$\cstar_r(H)$ again follows from the main result in~\cite{TWW}. Since $\cstar_r(G)$ and
$\cstar_r(H)$ are stably isomorphic we conclude that $\cstar_r(G)$ is quasidiagonal too.
\end{proof}

Thus stable finiteness for \cstar-algebras arising from minimal, \'{e}tale ample
groupoids $G$ is characterized by the `non-infinite' nature of the type semigroup $G$.
More precisely, if we call an element $\theta\in S(G)$ \emph{infinite} provided
$(n+1)\theta\leq n\theta$ for some $n\in\mathbb{N}$, then Theorem~\ref{theoremsfcnp} says
that $\cstar_r(G)$ is stably finite provided that $S(G)$ contains no infinite elements.
In the next section we shall look at the diametrically opposite setting where every
element in $S(G)$ is not only infinite, but `properly infinite'.

\section{Purely Infinite Groupoid \texorpdfstring{\cstar}{C*}-algebras}\label{sec:pi}

We now wish to characterize purely infinite reduced groupoid \cstar-algebras by the `properly infinite' nature of the corresponding type semigroup constructed above.  This coincides in spirit with the work of R{\o}rdam and Sierakowski in~\cite{RS} and of the first author in~\cite{Ra2}.

Recall that a projection $p$ in a \cstar-algebra $A$ is properly infinite if there are two subprojections $q,r\leq p$ with $qr=0$ and $q\sim p\sim r$. A unital \cstar-algebra $A$ is properly infinite if its unit $1_A$ is properly infinite. Purely infinite \cstar-algebras were introduced by J.\ Cuntz in~\cite{Cu}; an algebra $A$ is called purely infinite if every hereditary \cstar-subalgebra of $A$ contains a properly infinite
projection. It was a longstanding open question whether all unital, separable, simple, and nuclear  \cstar-algebra satisfied the stably finite/ purely infinite dichotomy until M. R{\o}rdam answered this query negatively in~\cite{Ro1}. He constructed a unital, simple, nuclear, and separable \cstar-algebra $D$ containing a finite projection $p$ and
an infinite projection $q$. It follows that $A=qDq$ is unital, separable, nuclear, simple, and properly infinite, but not purely infinite. It seems natural to ask if there is a smaller class of algebras for which a stably finite/purely infinite dichotomy holds. Theorem~\ref{dichotomy} below gives a partial answer in this direction.

We first present a necessary condition for a groupoid \cstar-algebra to be purely infinite.

\begin{proposition}
Let $G$ be an \'{e}tale groupoid. If $\cstar_r(G)$ is purely infinite, then  for any
non-empty compact open $U\subseteq\Gunit$, there is an open bisection $E\in\B$ with
$E\cap\Gunit=\emptyset$ such that
\[r(EU)\cap U\neq\emptyset.\]
In particular, for every non-empty compact open $U\subseteq\Gunit$, there is an
$\alpha\notin\Gunit$ with $r(\alpha), s(\alpha)\in U$.
\end{proposition}

\begin{proof}
Since $\cstar_r(G)$ is purely infinite the projection $p=\ch_U$ is properly infinite
(Theorem 4.16 in~\cite{KiR}). So there are $x,y\in\cstar_r(G)$ that satisfy
\[x^*x=p=y^*y,\quad xx^*\perp yy^*,\quad xx^*, yy^*\leq p.\]
The $\ast$-algebra $C_c(G)$ is norm-dense in $\cstar_r(G)$, so we may find sequences
$(a_n)_n, (b_n)_n$ in $C_c(G)$ converging to $x$ and $y$ respectively. We now compress by
setting $x_n:=pa_np$ and $y_n:=pb_np$ and note the following:
\[(x_n)_n\longrightarrow x\quad\text{since}\quad x_n=pa_np\longrightarrow pxp=pxx^*x=xx^*x=x,\]
\[(y_n)_n\longrightarrow y\quad\text{by a similar argument},\]
\[(x_n^*x_n)_n\longrightarrow p\quad\text{since}\quad x_n^*x_n\longrightarrow x^*x=p,\]
\[\text{similarly}\quad (y_n^*y_n)_n\longrightarrow p,\quad\text{and}\]
\[(x_n^*y_n)_n\longrightarrow 0\quad\text{since}\quad x_n^*y_n\longrightarrow x^*y=x^*xx^*yy^*y=0.\]
Moreover, the $x_n$ cannot be normal for all large $n$. To see why, suppose
$x_{n}^*x_{n}=x_{n}x_{n}^*$ for $n$ large, then we would have
\[p=p^2=\left(\lim_{n}x_n^*x_n\right)p=\left(\lim_{n}x_nx_n^*\right)p=xx^*p=xx^*\]
which contradicts the fact that $p$ is infinite. By passing to a subsequence we may
assume that all the $x_n$ are non-normal.

Since $G$ is \'{e}tale we can write $x_n=\sum_{k=1}^{K_n}f_{n,k}$ such that, for a fixed
$n$, the $f_{n,k}\in C_c(G)$ are non-zero and supported on distinct open bisections, say
$E_{n,k}$, with $E_{n,1}\subseteq\Gunit$ and $E_{n,k}\cap\Gunit=\emptyset$ for $2\leq k\leq
K_n$. Since $px_np=x_n$ we get
\[x_n=\sum_{k=1}^{K_n}f_{n,k}=\sum_{k=1}^{K_n}\ch_Uf_{n,k}\ch_U.\]
If $\ch_Uf_{n,k}\ch_U=0$ for $k=2,\dots, K_n$, then $x_n=\ch_Uf_{n,1}\ch_U$ is normal,
contradicting our assumption. Therefore, there is an open bisection, say $E$, with
$E\cap\Gunit=\emptyset$, and a non-zero $f\in C_c(G)$ supported in $E$ such that
$\ch_Uf\ch_U\neq 0$. But
\[\emptyset\neq\supp(\ch_Uf\ch_U)\subseteq UEU=\{\alpha\ |\ \alpha\in E, s(\alpha), r(\alpha)\in U\}.\]
It now follows that $r(EU)\cap U\neq\emptyset$.
\end{proof}

As alluded to above, it is the properly infinite nature of the type semigroup $S(G)$ that
will generate properly infinite projections in $\cstar_r(G)$. This is seen in
Lemma~\ref{pilemma} below. For this reason we introduce some terminology. An element
$\theta$ in a pre-ordered abelian group $S$ is said to \emph{properly infinite} if
$2\theta\leq\theta$, that is, if it is $(2,1)$-paradoxical, or equivalently it is
$(k,1)$-paradoxical for any $k\geq 2$. If every member of $S$ is properly infinite then
$S$ is said to be \emph{purely infinite}. A pre-ordered monoid $S$ is said to be
\emph{almost unperforated} if, whenever $\theta,\eta\in S$, and $n,m\in\mathbb{N}$ are
such that $n\theta\leq m\eta$ and $n>m$, then $\theta\leq\eta$.

\begin{lemma}\label{pilemma}
Let $G$ be an \'{e}tale and ample groupoid, and suppose $A\subseteq\Gunit$ is a compact
open subset. Let $\theta=[\ch_A]_G$ denote the class of $\ch_A$ in $S(G)$. The following
are equivalent.
\begin{enumerate}
\item[(i)] There are mutually disjoint compact open bisections
\[F_1,\dots, F_n, H_1,\dots, H_n\in\C\]
satisfying the following: writing $x=\sum_{i=1}^{n}\ch_{F_i}$ and
$y=\sum_{i=1}^{n}\ch_{H_i}$ in $\cstar_r(G)$,
\[x^*x=\ch_A,\quad y^*y=\ch_A,\quad xx^*+yy^*\leq\ch_A.\]
\item[(ii)] $A$ is $(k,1)$-paradoxical for some $k\geq2$.
\item[(iii)] $\theta$ is properly infinite in $S(G)$.
\end{enumerate}
\end{lemma}

\begin{proof}
$(i)\Rightarrow(ii)$: By assumption we have
\[\ch_A=x^*x=\bigg(\sum_{i=1}^{n}\ch_{F_i}\bigg)^{*}\bigg(\sum_{i=1}^{n}\ch_{F_i}\bigg)=\sum_{i,j}\ch_{F_i}^*\ch_{F_j}=\sum_{i,j}\ch_{F_i^{-1}}\ch_{F_j}=\sum_{i,j}\ch_{F_i^{-1}F_j}.\]
Applying the conditional expectation $\mathbb{E}$ on both sides we get
\[\ch_A=\mathbb{E}(\ch_A)=\mathbb{E}\bigg(\sum_{i,j}\ch_{F_i^{-1}F_j}\bigg)=\sum_{i,j}\mathbb{E}(\ch_{F_i^{-1}F_j})=\sum_{i,j}\ch_{F_i^{-1}F_j\cap\Gunit}=\sum_{i=1}^{n}\ch_{s(F_i)},\]
where we have used the fact that the $F_i$ are mutually disjoint to ensure that for
$i\neq j$, $F_i^{-1}F_j\cap\Gunit=\emptyset$, and $F_i^{-1}F_i=s(F_i)$. By a similar
calculation we also have
\[\ch_A=\sum_{i=1}^{n}\ch_{s(H_i)}.\]
Moreover, applying the conditional expectation to
\begin{align*}\ch_A\geq xx^*+yy^*&=\bigg(\sum_{i=1}^{n}\ch_{F_i}\bigg)\bigg(\sum_{i=1}^{n}\ch_{F_i}\bigg)^*+\bigg(\sum_{i=1}^{n}\ch_{H_i}\bigg)\bigg(\sum_{i=1}^{n}\ch_{H_i}\bigg)^*\\&=\sum_{i,j}\ch_{F_iF_j^{-1}}+\sum_{i=i,j}\ch_{H_iH_j^{-1}}
\end{align*}
gives
\begin{align*}
\ch_A&=\mathbb{E}(\ch_A)\geq\mathbb{E}\bigg(\sum_{i,j}\ch_{F_iF_j^{-1}}+\sum_{i,j}\ch_{H_iH_j^{-1}}\bigg)=\sum_{i,j}\ch_{F_iF_j^{-1}\cap\Gunit}+\sum_{i=i,j}\ch_{H_iH_j^{-1}\cap\Gunit}\\&=\sum_{i=1}^{n}\ch_{F_iF_i^{-1}}+\sum_{i=1}^{n}\ch_{H_iH_i^{-1}}=\sum_{i=1}^{n}\ch_{r(F_i)}+\sum_{i=1}^{n}\ch_{r(H_i)}.
\end{align*}
Again here we use disjointness so that
$F_iF_j^{-1}\cap\Gunit=\emptyset=H_iH_j^{-1}\cap\Gunit$ for $i\neq j$. All together,
\[2\ch_A=\sum_{i=1}^{n}\ch_{s(F_i)}+\sum_{i=1}^{n}\ch_{s(H_i)},\quad\text{and}\quad\sum_{i=1}^{n}\ch_{r(F_i)}+\sum_{i=1}^{n}\ch_{r(H_i)}\leq\ch_A\]
which says that $\ch_A$ is $(2,1)$-paradoxical.

$(ii)\Rightarrow(i)$: We may suppose that $k=2$ since $2\ch_A\leq k\ch_A$. There are,
therefore, compact open bisections $E_1,\dots, E_n$ in $G$ with
\[2\ch_A\leq\sum_{i=1}^{n}\ch_{s(E_i)},\quad\text{and}\quad\sum_{i=1}^{n}\ch_{r(E_i)}\leq\ch_A.\]
This condition implies that the $r(E_i)$ are mutually disjoint and therefore the
bisections $E_i$ are themselves mutually disjoint.

The inequality $2\ch_A\leq\sum_{i=1}^{n}\ch_{s(E_i)}$ implies that by partitioning copies
of $A$, we can find compact open sets $\{A_{i,j}\ |\ 1\leq i\leq n, 1\leq j\leq 2\}$ such
that
\[\bigsqcup_{i=1}^{n}A_{i,1}=A,\quad \bigsqcup_{i=1}^{n}A_{i,2}=A,\quad A_{i,1}\sqcup A_{i,2}\subseteq s(E_i)\quad\forall i\in\{1,\dots,n\}. \]
Now for each $i=1,\dots, n$ set
\[F_i:=(s|_{E_i})^{-1}({A_{i,1}})\quad\text{and}\quad H_i:=(s|_{E_i})^{-1}({A_{i,2}}).\]
Note the following:
\[F_i\subseteq E_i,\quad s(F_i)=A_{i,1},\quad i\neq j\Rightarrow r(F_i)\cap r(F_j)=\emptyset,\]
\[H_i\subseteq E_i,\quad s(H_i)=A_{i,2},\quad i\neq j\Rightarrow r(F_i)\cap r(F_j)=\emptyset.\]
Moreover, $F_i\cap H_i=\emptyset$ for every $i$, and since the $E_i$ are disjoint, it
follows that the compact bisections $F_1,\dots, F_n, H_1,\dots, H_n$ are mutually
disjoint.

Using these facts and writing $x$ and $y$ as in the statement of the lemma we get
\[x^*x=\bigg(\sum_{i=1}^{n}\ch_{F_i}\bigg)^{*}\bigg(\sum_{i=1}^{n}\ch_{F_i}\bigg)=\sum_{i,j}\ch_{F_i^{-1}F_j}=\sum_{i=1}^{n}\ch_{s(F_i)}=\sum_{i=1}^{n}\ch_{A_{i,1}}=\ch_A,\]
\[y^*y=\bigg(\sum_{i=1}^{n}\ch_{H_i}\bigg)^{*}\bigg(\sum_{i=1}^{n}\ch_{H_i}\bigg)=\sum_{i,j}\ch_{H_i^{-1}H_j}=\sum_{i=1}^{n}\ch_{s(H_i)}=\sum_{i=1}^{n}\ch_{A_{i,2}}=\ch_A,\]
where we use the fact that $F_i^{-1}F_j=\emptyset=H_i^{-1}H_j=\emptyset$ for $i\neq j$
since these have disjoint ranges. Now using the fact that $s(F_i)\cap
s(F_j)=\emptyset=s(H_i)\cap s(H_j)$ for $i\neq j$, and that $E_i$ are bisections we get
\begin{align*}xx^*+yy^*&=\sum_{i,j}\ch_{F_iF_j^{-1}}+\sum_{i,j}\ch_{H_iH_j^{-1}}=\sum_{i=1}^{n}\ch_{r(F_i)}+\sum_{i=1}^{n}\ch_{r(H_i)}\\&=\sum_{i=1}^{n}(\ch_{r(F_i)}+\ch_{r(H_i)})\leq\sum_{i=1}^{n}\ch_{r(E_i)}\leq\ch_A,
\end{align*}
thus $(i)$ holds.

The equivalence $(ii)\Leftrightarrow(iii)$ follows directly from
Lemma~\ref{lemmaparadoxtype}.
\end{proof}

We now come to the main result of this section. We use the constructed type semigroup
$S(G)$ to characterize purely infinite \cstar-algebras that arise from ample groupoids.
The following theorem is in the same spirit of Theorem 5.4 of~\cite{RS} and Theorem 5.6
of~\cite{Ra2}.

\begin{theorem}\label{pitheorem} Let $G$ be a topological groupoid which is \'{e}tale, ample, minimal, and topologically principal. Consider the following properties:
\begin{enumerate}
\item[(i)] The semigroup $S(G)$ is purely infinite.
\item[(ii)] Every non-empty compact open $A\subseteq\Gunit$ is properly paradoxical.
\item[(iii)] The $\cstar$-algebra $\cstar_r(G)$ is purely infinite.
\item[(iv)] The $\cstar$-algebra $\cstar_r(G)$ admits no tracial state.
\item[(v)] The semigroup $S(G)$ admits no non-trivial state.
\end{enumerate}
Then the following implications always hold:
$(i)\Leftrightarrow(ii)\Rightarrow(iii)\Rightarrow(iv)\Rightarrow(v)$. If the semigroup
$S(G)$ is almost unperforated then $(v)\Rightarrow(i)$ and all properties are equivalent.
\end{theorem}

\begin{proof}
$(i)\Rightarrow(ii)$: Let $A\subseteq\Gunit$ be compact and open. Since we are assuming that
$\theta:=[\ch_A]_G$ is properly infinite, we have $2\theta\leq\theta$.
Lemma~\ref{lemmaparadoxtype} now says that $A$ is $(2,1)$-paradoxical whence properly paradoxical.

$(ii)\Rightarrow(i)$: Let $\theta=[f]_G$ be a non-zero element in $S(G)$. We may write $f=\sum_k\ch_{A_k}$ where $A_k\subseteq\Gunit$ are non-zero compact open subsets. Setting $\theta_k:=[\ch_{A_k}]_G$, again Lemma~\ref{lemmaparadoxtype} implies that $2\theta_k\leq\theta_k$ for each $k$. The quotient map is additive so $\sum_k\theta_k=\theta$. It easily follows that $\theta$ is properly paradoxical since
\[2\theta=2\sum_k\theta_k=\sum_k2\theta_k\leq\sum_{k}\theta_k=\theta.\]

$(ii)\Rightarrow(iii)$: By the main result in~\cite{BCS} it suffices to check that every
non-empty compact open set $A\subseteq\Gunit$ defines a properly infinite projection
$\ch_A$ in $\cstar_r(G)$. By our assumption every such $A$ is $(2,1)$-paradoxical, and
Lemma~\ref{pilemma} implies that $\ch_A$ is properly infinite.

$(iii)\Rightarrow(iv)$: Purely infinite algebras are always traceless.

$(iv)\Rightarrow(v)$: Suppose $\nu:S(G)\rightarrow[0,\infty]$ is a non-trivial state.
Suppose $0<\nu([\ch_{U}]_G)<\infty$ where $U\subseteq\Gunit$ is a non-empty compact open
subset. Composing with the quotient  map $\pi:C(\Gunit,\mathbb{Z})^+\rightarrow S(G)$ we
get an order preserving monoid homomorphism
$\beta=\nu\circ\pi:C(\Gunit,\mathbb{Z})^+\rightarrow[0,\infty]$ with
$0<\beta(\ch_U)<\infty$. As in the proof of Proposition~\ref{theoremsfcnp}, minimality of
$G$ ensures that $\beta$ is finite on all of $C(\Gunit,\mathbb{Z})^+$. Extending $\beta$
to $C(\Gunit,\mathbb{Z})$ gives an invariant positive group homomorphism,
$\beta:C(\Gunit,\mathbb{Z})\rightarrow\mathbb{R}$. Identifying
$C(\Gunit,\mathbb{Z})\cong\Ko(C(\Gunit))$ as in the proof of Theorem~\ref{theoremsfcnp}
we see that $\beta$ comes from a $G$-invariant Borel probability measure which induces a
tracial state on $\cstar_r(G)$, a contradiction.

Now we assume that $S(G)$ is almost unperforated and prove $(v)\Rightarrow(i)$. Let
$\theta$ be a non-zero element in $S(G)$. If $\theta$ is completely non-paradoxical then
by Tarski's Theorem $S(G)$ admits a non-trivial state. So, assuming $(v)$, we must have
$(k+1)\theta\leq k\theta$ for some $k\in\mathbb{N}$. So
\[(k+2)\theta=(k+1)\theta+\theta\leq k\theta+\theta=(k+1)\theta\leq k\theta.\]
Repeating this we get $(k+1)2\theta\leq k\theta$. Since $S(G)$ is almost unperforated we
conclude $2\theta\leq\theta$ and $\theta$ is properly infinite.
\end{proof}

As promised, we combine Theorem~\ref{theoremsfcnp}, Corollary~\ref{corollarysfcnp} and Theorem~\ref{pitheorem} to obtain the long-desired
dichotomy.

\begin{theorem}\label{dichotomy}
Let $G$ be an \'{e}tale, ample, minimal, and topologically principal groupoid with an almost unperforated type semigroup $S(G)$.
\begin{enumerate}
\item[(i)] The \cstar-algebra $\cstar_r(G)$ is simple and is either purely infinite or stably finite.
\item [(ii)] If $G$ is also amenable, then $\cstar_r(G)$ is simple and is either purely infinite or quasidiagonal.
\end{enumerate}
 \end{theorem}

\begin{remark}
Since, by Remark~\ref{rmk:BL matchup}, our type semigroup coincides with that of
B\"onicke and Li \cite[Definition~5.3]{BL}, our Theorems
\ref{pitheorem}~and~\ref{dichotomy} recover \cite[Theorem~5.11]{BL}, and extend it to
groupoids with non-compact unit spaces. In particular, our Theorem~\ref{dichotomy}
extends \cite[Corollary~5.13]{BL} to groupoids with not-necessarily-compact unit spaces.
\end{remark}

\section{Crossed products and \texorpdfstring{$k$}{k}-graphs}\label{sec:reconcile}

In this section we reconcile our semigroup $S(G)$ with two previous constructions
appearing in the literature.

Firstly, if $X$ is a compact totally disconnected Hausdorff space carrying an action of a
discrete group $\Gamma$, then we can form the transofmation groupoid. The reduced
$\cstar$-algebra of the transformation groupoid $G$ coincides with the reduced crossed
product of $C(X)$ by $\Gamma$, so we would expect our semigroup $S(G)$ to coincide with
the type semigroup of the induced action of $\Gamma$ on $C(X)$ defined and studied in
\cite{RS, Ra2}. We prove that this is the case in Proposition~\ref{generalization} (this
is also discussed in \cite[Remark 5.4.]{BL}); we will use this in Section~\ref{sec:orbit
equiv} to see that the type semigroup of the action of $\Gamma$ on $X$ is an
orbit-equivalence-invariant.

Secondly, if $\Lambda$ is a row-finite $k$-graph with no sources, its $\cstar$-algebra
can be realised as a crossed product of an AF algebra by $\mathbb{Z}^k$. The type
semigroup $S(\Lambda)$ of this action can be computed directly in terms of the adjacency
matrices of the $k$-graph \cite{PSS}. Since $k$-graph $\cstar$-algebras can also be
modelled as the $\cstar$-algebras of ample groupoids, it is natural to compare the
semigroup $S(\Lambda)$ with the type semigroup of the associated groupoid. We prove in
Proposition~\ref{prp:k-graph S} that the two are isomorphic.

\begin{proposition}\label{generalization}
Let $\Gamma\curvearrowright X$ be an action of a discrete group on a compact and totally
disconnected space $X$, and let $G=\Gamma\rtimes X$ denote the resulting ample, \'{e}tale
transformation groupoid. Then $S(G)$ and $S(X,\Gamma)$ (as constructed in~\cite{RS}) are
isomorphic as preordered abelian monoids.
\end{proposition}

\begin{proof}
Proposition 4.4 in~\cite{Ra2} shows that $S(X,\Gamma)\cong S(C(X),\Gamma)$, so it
suffices to show that $S(G)\cong S(C(X),\Gamma)$. In what follows we may, and will,
identify the spaces $\Gunit$ and $X$. We thus consider the identity mapping
$C(\Gunit,\mathbb{Z})^+\rightarrow C(X,\mathbb{Z})^+$ and prove that $f\sim_{G} g$ if and
only if $f\sim_{\alpha} g$ (in the sense of definition 4.1 of~\cite{Ra2}) where $f,g\in
C(X,\mathbb{Z})^+$ and $\alpha:\Gamma\curvearrowright C(X)$ is the induced action.

Let $E\subseteq G$ be a compact open bisection. Writing $\pi_\Gamma:G\rightarrow\Gamma$ and
$\pi_X:G\rightarrow X$ for the canonical projections, we set $E_X=\pi_X(E)=s(E)\subseteq X$
and $E_\Gamma=\pi_{\Gamma}(E)$. Note that for each $x\in E_X$ there is a unique $t_x$ in
$E_\Gamma$ such that $(t_x,x)\in E$. For $t\in E_\Gamma$, let
\[E_{X}(t)=\{x\in E_X\ |\ t_x=t\}=s(E\cap\pi_{\Gamma}^{-1}(\{t\})).\]
Clearly $E_X(t)$ is compact and open and since $E=\bigsqcup_{t\in E_{\Gamma}}(E\cap
\pi_{\Gamma}^{-1}(\{t\})$ and $s$ is bijective on $E$ we can see that $\bigsqcup_{t\in
E_{\Gamma}}E_X(t)=E_X$. By compactness there are finitely many $t_j\in E_\Gamma$, $1\leq
j\leq m$ such that \[\bigsqcup_{j=1}^{m}E_X(t_j)=E_X=s(E).\] By construction it also
follows that
\[\bigsqcup_{j=1}^{m}t_j.E_X(t_j)=r(E).\]
Setting $E_j=E_X(t_j)$ for $j=1,\dots,m$ we get
\[\ch_{s(E)}=\sum_{j=1}^{m}\ch_{E_j},\quad\text{and}\quad \ch_{r(E)}=\sum_{j=1}^{m}\ch_{t_j.E_j}.\]

Now suppose $f,g\in C(\Gunit,\mathbb{Z})^+$ with $f\sim_G g$. Then there are compact open
bisections $E_1,\dots, E_n$ with
\[f=\sum_{i=1}^{n}\mathds{1}_{s(E_i)},\quad \text{and}\quad g=\sum_{i=1}^{n}\mathds{1}_{r(E_i)}.\]
By our work above, for each $i=1,\dots, n$ we can find an $m_i$, group elements
$t_{i,1},\dots, t_{i,m_i}$, and compact open subsets $E_{i,1},\dots,E_{i,m_{i}}\subseteq
s(E_i)$ with
\[\ch_{s(E_i)}=\sum_{j=1}^{m_i}\ch_{E_{i,j}}\quad \text{and}\quad \ch_{r(E_i)}=\sum_{j=1}^{m_i}\ch_{t_{i,j}E_{i,j}}.\]
It then follows that $f\sim_{\alpha}g$ because
\[f=\sum_{i=1}^{n}\sum_{j=1}^{m_i}\ch_{E_{i,j}}\quad \text{and}\quad g=\sum_{i=1}^{n}\sum_{j=1}^{m_i}\ch_{t_{i,j}E_{i,j}}.\]

Conversely, suppose $f,g\in C(X,\mathbb{Z})^+$ with $f\sim_\alpha g$. Then there are
compact open subsets $A_1,\dots, A_n\subseteq X$ and group elements $t_1,\dots,
t_n\in\Gamma$ such that $f=\sum_{i=1}^{n}\ch_{A_i}$ and $g=\sum_{i=1}^{n}\ch_{t_i.A_i}$.
Simply set $E_i=\{t_i\}\times A_i$, which are clearly compact open bisections in $G$.
Then $s(E_i)=A_i$ and $r(E_i)=t_i.E_i$
\[f=\sum_{i=1}^{n}\ch_{A_i}=\sum_{i=1}^{n}\ch_{s(E_i)},\quad\text{and}\quad g=\sum_{i=1}^{n}\ch_{t_i.A_i}=\sum_{i=1}^{n}\ch_{r(E_i)}\]
which means $f\sim_G g$.
\end{proof}

We now reconcile our type semigroup with the semigroup of a $k$-graph constructed in
\cite{PSS}. Given a row-finite $k$-graph $\Lambda$ with no sources, the associated
semigroup $S(\Lambda)$ is defined \cite[Definition~3.5]{PSS} as follows. For $n \in
\mathbb{N}^k$, we write $A_\Lambda^n$ for the $\Lambda^0 \times \Lambda^0$ integer matrix
with entries $A_\Lambda^n(v,w) = |v\Lambda^n w|$. The semigroup $S(\Lambda)$ is defined
to be the quotient of $\mathbb{N}\Lambda^0$ by the equivalence relation $\approx$ defined
as follows: we first write $x \sim y$ if there exist $p,q \in \mathbb{N}^k$ such that
$(A^p_\Lambda)^t x = (A^q_\Lambda)^t y$; and then $x \approx y$ if there exist finitely
many pairs $(x_i, y_i)$ in $\mathbb{N}\Lambda^0$ such that
\[
\sum_i x_i \sim x,\quad \sum_i y_i \sim y\quad\text{ and }\quad x_i \sim y_i\text{ for all $i$.}
\]

In \cite[Lemma~3.7]{PSS} the semigroup $S(\Lambda)$ is related to the type semigroup of
an associated $\cstar$-dynamical system: each twisted $\cstar$-algebra $\cstar(\Lambda,
c)$ of $\Lambda$ is realised, up to stable isomorphism, as a crossed product of an AF
algebra, and \cite[Lemma~3.7]{PSS} shows that the type semigroup for this dynamical
system is isomorphic to $S(\Lambda)$.

Kumjian and Pask \cite{KP} showed that every $k$-graph has an associated infinite-path
groupoid $G_\Lambda$ such that $\cstar(\Lambda) \cong \cstar(G_\Lambda)$. Here we prove
that $S(\Lambda)$ agrees with the semigroup $S(G_\Lambda)$ if the infinite-path groupoid
of $\Lambda$.

We need to briefly recall the notion of a $k$-graph and the definition of $G_\Lambda$.
The following is all taken from \cite{KP}. A $k$-graph is a countable category $\Lambda$
endowed with a map $d : \Lambda \to \mathbb{N}^k$ that carries composition to addition
and has the property, called the \emph{factorisation property} that composition restricts
to a bijection from $\{(\mu,\nu) : d(\mu) = m, d(\nu) = n, s(\mu) = r(\nu)\}$ to
$d^{-1}(m+n)$. It follows that $d^{-1}(0)$ is precisely the collection of identity
morphisms. We write $\Lambda^n := d^{-1}(n)$ for $n \in \mathbb{N}^k$, so that $r,s$ can
be regarded as maps from $\Lambda$ to $\Lambda^0$. We say that $\Lambda$ is row-finite
with no sources if $v \Lambda^n$ is finite and nonempty for every $n \in \mathbb{N}^k$
and $v \in \Lambda^0$.

An \emph{infinite path} in a $k$-graph $\Lambda$ is a map $x : \{(m,n) \in \mathbb{N}^k
\times \mathbb{N}^k : m \le n\} \to \Lambda$ such that $x(m,n)x(n,p)$ is defined and
equal to $x(m,p)$ whenever $m \le n \le p$. The space $\Lambda^\infty$ of all such
infinite paths is a totally disconnected locally compact space under the topology with
basic compact open sets $Z(\lambda) := \{x : x(0,d(\lambda)) = \lambda\}$ indexed by
$\lambda \in \Lambda$. Given $x \in \Lambda^\infty$ and $n \in \mathbb{N}^k$ we write
$\sigma^n(x) \in \Lambda^\infty$ for the element such that $\sigma^n(x)(p,q) = x(n+p,
n+q)$. The maps $\sigma^n$ constitute an action of $\mathbb{N}^k$ on $\Lambda^\infty$ by
local homeomorphisms. The groupoid $G_\Lambda$ is the set
\[
G_\Lambda := \{(x,p-q,y) : x,y \in \Lambda^\infty, p,q \in \mathbb{N}^k, \sigma^p(x) = \sigma^q(y)\},
\]
with composable pairs $G_\Lambda^{(2)} = \{((x,m,y),(w,n,z)) : y = w\}$, composition
given by $(x,m,y)(y,n,z) = (x, m+n, z)$ and inverses $(x,m,y)^{-1} = (y, -m, x)$. The
unit space is $G_\Lambda^{(0)} = \{(x,0,x) : x \in \Lambda^\infty\}$, and we identify it
with $\Lambda^\infty$ without further comment. Under the topology generated by the sets
$Z(\mu,\nu) := \{(x, d(\mu) - d(\nu), y) : x(0,d(\mu)) = \mu\text{ and }x(0, d(\nu))=
\nu\}$ indexed by pairs $\mu,\nu \in \Lambda$ with $s(\mu) = s(\nu)$, the groupoid
$G_\Lambda$ is an ample \'etale amenable groupoid, and the sets $Z(\mu,\nu)$ are compact
open bisections.

\begin{proposition}\label{prp:k-graph S}
Let $\Lambda$ be a row-finite $k$-graph with no sources, and let $G_\Lambda$ be the
associated $k$-graph groupoid. Then there is an isomorphism $\tau : S(\Lambda) \cong
S(G_\Lambda)$ such that $\tau([\delta_v]) = [\ch_{Z(v)}]$ for all $v \in \Lambda^0$.
\end{proposition}
\begin{proof}
There is a homomorphism $\tilde\tau : \mathbb{N}\Lambda^0 \to C_c(G_\Lambda,
\mathbb{Z})_+$ that carries $\delta_v$ to $\ch_{Z(v)}$ for all $v$. For $v \in \Lambda^0$
and $p \in \mathbb{N}^k$, we have
\[
Z(v) = \bigsqcup_{\lambda \in v\Lambda^p} Z(\lambda)
    = \bigsqcup_{\lambda \in v\Lambda^p} r(Z(\lambda, s(\lambda)))
\]
Hence
\begin{align*}
\tilde\tau(\delta_v)
    &= \ch_{Z(v)}
    = \sum_{\lambda \in v\Lambda^p} \ch_{r(Z(\lambda, s(\lambda)))}\\
    &\sim \sum_{\lambda \in v\Lambda^p} \ch_{s(Z(\lambda, s(\lambda)))}
    = \sum_{\lambda \in v\Lambda^p} \ch_{Z(s(\lambda))}
    = \sum_{\lambda \in v\Lambda^p} \tilde\tau(\delta_{s(\lambda)})
    = \tilde\tau\big((A_\Lambda^p)^t\delta_v\big).
\end{align*}
A simple calculation then shows that if $x \approx y$ in $\mathbb{N}\Lambda^0$ then
$\tilde\tau(x) \sim \tilde\tau(y)$ in $C_c(G_\Lambda, \mathbb{Z})_+$ and so $\tilde\tau$
descends to a homomorphism from $S(\Lambda)$ to $S(G_\Lambda)$. To see that this $\tau$
is surjective, it suffices to show that its range contains $[\ch_K]$ for every compact
open $K \subseteq G_\Lambda^0 = \Lambda^\infty$. To see this, fix such a compact open
$K$. Since the cylinder sets $\{Z(\lambda) : \lambda \in \Lambda\}$ are a base for the
topology on $\Lambda^\infty$, for each $x \in K$ we can find $\lambda_x$ such that $x \in
Z(\lambda_x) \subseteq K$. By compactness, we there is a finite $F \subseteq \Lambda$
such that $K = \bigcup_{\lambda \in F} Z(\lambda)$. Let $p := \bigvee_{\lambda \in F}
d(\lambda)$. Then each $Z(\lambda) = \bigsqcup_{\alpha \in s(\lambda)\Lambda^{p -
d(\lambda)}} Z(\lambda\alpha)$. Let $\overline{F} := \{\lambda\alpha : \lambda \in F,
\alpha \in s(\lambda)\Lambda^{p - d(\lambda)}\}$. Then $K = \bigcup_{\mu \in
\overline{F}} Z(\mu)$. Since the sets $\{Z(\mu) : \mu \in \Lambda^p\}$ are mutually
disjoint, we conclude that $K = \bigsqcup_{\mu \in G} Z(\mu)$. Hence
\[
[\ch_K]
    = \Big[\sum_{\mu \in \overline{F}} \ch_{r(Z(\mu, s(\mu)))}\Big]
    = \Big[\sum_{\mu \in \overline{F}} \ch_{s(Z(\mu), s(\mu))}\Big]
    = \tau\Big(\sum_{\mu \in \overline{F}} \delta_{s(\mu)}\Big).
\]

To see that $K$ is injective, we show that if $\tilde\tau(x) \sim \tilde\tau(y)$, then $x
\approx y$. Since $\tilde\tau(x) \sim \tilde\tau(y)$, then we can find compact open
bisections $E_i$ such that $\sum_v x(v)\ch_{Z(v)} = \sum_i \ch_{s(E_i)}$ and $\sum_w
y(w)\ch_{Z(w)} = \sum_i \ch_{r(E_i)}$. Recall that the map $c : (x, m, y) \mapsto m$ is a
continuous $\mathbb{Z}^k$-valued cocycle on $G_\Lambda$ and so the sets $\{c^{-1}(p) : p
\in \mathbb{Z}^k\}$ are mutually disjoint clopen sets. So by replacing each $E_i$ with
the finitely many nonempty intersections $E_i \cap c^{-1}(p)$ of which it is comprised,
we can assume that each $E_i \subseteq c^{-1}(p_i)$ for some $i$. Now an argument very
similar to the preceding paragraph shows that we can express each $E_i$ as a finite
disjoint union $E_i = \bigcup_{(\mu,\nu) \in F_i} Z(\mu,\nu)$, so we may assume that each
$E_i$ has the form $Z(\mu_i,\nu_i)$. By taking $p = \sup_i d(\mu_i)$ and writing each
$Z(\mu_i, \nu_i) = \bigsqcup_i \bigsqcup_{\alpha \in s(\mu_i)\Lambda^{p - d(\lambda_i)}}
Z(\mu_i\alpha,\nu_i\alpha)$, we can further assume that each $d(\mu_i) = p$. So the sets
$r(Z(\mu_i, \nu_i))$ and $r(Z(\mu_j, \nu_j))$ are either equal or disjoint for all pairs
$i,j$. Now for each $v \in \Lambda^0$ such that $x(v) \not= 0$ and each $\lambda \in
v\Lambda^p$, we must have $|\{i : \mu_i = \lambda\}| = x(v)$. We deduce that $\sum_i
\delta_{s(\mu_i)} = (A^p_\Lambda)^t x \approx x$.

Now let $q = \bigvee_i d(\nu_i)$. For each $i$, we have
\[
\delta_{s(\nu_i)}
    \approx (A_\Lambda^{q - d(\nu_i)})^t \delta_{s(\nu_i)}
    = \sum_{\alpha \in s(\nu_i)\Lambda^{q - d(\nu_i)}} \delta_{s(\nu_i\alpha)}.
\]
Thus
\[
(A^p_\Lambda)^t x
    \approx \sum_i \sum_{\alpha \in s(\nu_i)\Lambda^{q - d(\nu_i)}} \delta_{s(\nu_i\alpha)} =:z.
\]
As above, we have $\sum_i \sum_\alpha \ch_{Z(\nu_i\alpha)} = \tilde\tau(y)$, and since
each $d(\nu_i\alpha) = q$ the sets $Z(\nu_i\alpha)$ are mutually disjoint. So for each $w
\in \Lambda^0$ with $y(w) \not= 0$ and each $\eta \in w\Lambda^q$, we have $|\{(i,\alpha)
: \nu_i\alpha= \eta\}| = y(w)$. Hence
\[
z = (A_\lambda^q)^t y \approx y.
\]
Since $\approx$ is transitive, we conclude that $x \approx y$. So $\tau$ is injective.
\end{proof}

\section{Applications}\label{sec:applications}

In this section we present some applications of our results above.

\subsection{Dynamical systems and continuous orbit equivalence}\label{sec:orbit equiv}

We first apply our results to the theory of rigidity of dynamical systems. Recall the notion of continuous orbit equivalence studied by Giordano, Putnam, and Skau in~\cite{GPS}, and by X.\ Li in~\cite{Li}.

Let $(X,\Gamma)$ and $(Y,\Lambda)$ be discrete transformation groups and write $X\rtimes\Gamma$ and $Y\rtimes\Lambda$ for the resulting \'{e}tale transformation groupoids. The actions $\Gamma\curvearrowright X$ and $\Lambda\curvearrowright Y$ are said to be \emph{continuous orbit equivalent} (COE) if there are
\begin{enumerate}
\item[(i)] a homeomorphism $\varphi:X\rightarrow Y$ (with inverse $\psi:Y\rightarrow
    X$),
\item[(ii)] continuous maps $a:\Gamma\times X\rightarrow\Lambda$, and  $b:\Lambda\times Y\rightarrow \Gamma$
\end{enumerate}
such that for all $x\in X$, $t\in\Gamma$, $y\in Y$, $s\in\Lambda$
\[\varphi(t.x)=a(t,x)\varphi(x),\quad\text{and}\quad\psi(s.y)=b(s,y)\psi(y).\]
X.\ Li proves the following rigidity property for free systems.

\begin{theorem}[\cite{Li}]\label{Li}
Let $(X,\Gamma)$ and $(Y,\Lambda)$ be topologically free dynamical systems. The following
are equivalent.
\begin{enumerate}
\item[(i)] $(X,\Gamma)$ and $(Y,\Lambda)$ are continuous orbit equivalent.
\item[(ii)] The transformation groupoids $X\rtimes\Gamma$ and $Y\rtimes\Lambda$ are
    topologically isomorphic.
\item[(iii)] There is a \cstar-isomorphism $\Phi:C_0(X)\rtimes\Gamma\rightarrow
    C_0(Y)\rtimes\Lambda$ with $\Phi(C_0(X))=C_0(Y)$.
\end{enumerate}
\end{theorem}

The proof of $(i)\Rightarrow(ii)$ relies on topological freeness. Indeed, following the notation in the definition above of COE the maps
\begin{align*}
X\rtimes\Gamma\rightarrow &Y\rtimes\Lambda,\quad (t,x)\mapsto(a(t,x),\varphi(x))\\
Y\rtimes\Lambda\rightarrow &X\rtimes\Gamma,\quad (s,y)\mapsto(b(s,y),\psi(y))
\end{align*}
are easily seen to be topological groupoid homomorphisms. Topological freeness guarantees that these maps are mutual inverses.

From this result and our work above we can show that the type semigroup is a continuous orbit equivalence invariant under the assumption of freeness. Recall that for a Cantor system $(X,\Gamma)$, we write $S(X,\Gamma)$ for the type semigroup as defined in~\cite{RS}.

\begin{theorem}
Let $X$ and $Y$ be totally disconnected spaces and suppose $\Gamma\curvearrowright X$ and
$\Lambda\curvearrowright Y$ are topologically free dynamical systems. If $(X,\Gamma)$ and
$(Y,\Lambda)$ are continuous orbit equivalent then $S(X,\Gamma)\cong S(Y,\Lambda)$ as
abelian monoids.
\end{theorem}

\begin{proof}
Theorem~\ref{Li} says that $X\rtimes\Gamma\cong Y\rtimes\Lambda$ are topologically
isomorphic groupoids. Since the underlying spaces are totally disconnected these
groupoids are ample. Now apply Proposition~\ref{isomorphisminvariant} to get
$S(X\rtimes\Gamma)\cong S(Y\rtimes\Lambda)$. Finally, appealing to
Proposition~\ref{generalization} gives us the desired result.
\end{proof}

\subsection{\texorpdfstring{$n$}{n}-filling groupoids and pure infiniteness}

Due to the deep classification achievements of Kirchberg~\cite{Kir} and
Phillips~\cite{Phi}, several authors have sought to express dynamical conditions for
systems $(A,\Gamma)$ that yield purely infinite crossed products $A\rtimes_{r}\Gamma$.
For example, a continuous action $\Gamma\curvearrowright X$ of a discrete group on a
locally compact Hausdorff space is called a \emph{local boundary action} if for every
non-empty open $U\subseteq X$ there is an open $V\subseteq U$ and $t\in\Gamma$ such that
$t.\overline{V}\varsubsetneqq V$. Laca and Spielberg showed in~\cite{LS} that if
$\Gamma\curvearrowright X$ is a topologically free local boundary action  the reduced
crossed product $C_0(X)\rtimes_{r}\Gamma$ is purely infinite and simple. Jolissaint and
Robertson~\cite{JR} generalized this notion and defined an \emph{$n$-filling} property
which in the commutative setting is equivalent to the property:  for every collection
$U_1,\dots,U_n\subseteq X$ of non-empty open sets we can find group elements
$t_1,\dots,t_n$ in $\Gamma$ such that $\bigcup_{i=1}^{n}t_j.U_j=X$. This $n$-filling
property was subsequently generalised to groupoids by Suzuki \cite{Suzuki}, who proved
that $n$-filling topologically principal groupoids have purely infinite reduced
$C^*$-algebras. We show here how to recover Suzuki's pure-infiniteness result from our
notion of paradoxicality.

The following definition looks slightly different from that given by Suzuki in
\cite[Definition~3.2]{Suzuki}, but we show below that they are equivalent in
Remark~\ref{rmk:filling equiv}

\begin{definition}
Let $G$ be an \'{e}tale and ample groupoid, and let $n\in\mathbb{N}$. We say that $G$ has
the $n$-filling property if for every collection $U_1,\dots,U_n\subseteq\Gunit$ of
non-empty open sets we can find compact open bisections $E_1,\dots, E_n$ such that
\[
    \bigcup_{i=1}^{n}r(E_jU_j)=\Gunit.
\]
\end{definition}

\begin{proposition} \label{prp:filling->pi}
Let $G$ be an \'{e}tale and ample groupoid with compact unit space $\Gunit$ void of
isolated points. If $G$ is $n$-filling, then $G$ is minimal and every non-empty compact
open $A\subseteq\Gunit$ is properly paradoxical.

If, moreover, $G$ is topologically principal, then $\cstar_r(G)$ is purely infinite.
\end{proposition}
\begin{proof}
Minimality is clear by Proposition~\ref{min2}. Suppose $A\subseteq\Gunit$ is a non-empty
compact open subset. Since $\Gunit$ is totally disconnected without isolated points we
can find $2n$ disjoint clopen sets $U_1,\dots, U_{2n}\subseteq A$. By the filling property
we can find compact open bisections $E_1,\dots, E_{2n}$ such that
\[\bigcup_{j=1}^{n}r(E_jU_j)=\Gunit,\quad\text{and}\quad \bigcup_{j=n+1}^{2n}r(E_jU_j)=\Gunit.\]
Consider the compact open bisections $F_j=(E_jU_j)^{-1}$ for $j=1,\dots, 2n$. We note
that
\[\sum_{j=1}^{2n}\ch_{s(F_j)}=\sum_{j=1}^{2n}\ch_{r(E_jU_j)}\geq\ch_{\bigcup_{j=1}^{n}r(E_jU_j)}+\ch_{\bigcup_{j=n+1}^{2n}r(E_jU_j)}\geq\ch_{\Gunit}+\ch_{\Gunit}=2\ch_{\Gunit}.\]
Since $r(F_j)=s(E_jU_j)\subseteq U_j$ for all $j$, and these are disjoint, we have
\[\sum_{j=1}^{2n}\ch_{r(F_j)}\leq \sum_{j=1}^{2n}\ch_{U_j}=\ch_{\bigcup_{j=1}^{2n}U_j}\leq\ch_{A},\]
thus $A$ is properly paradoxical.

The final assertion follows from Theorem~\ref{pitheorem}.
\end{proof}

It would be reasonable to suspect that this notion of filling is related to the locally
contracting property of C.\ Anantharaman-Delaroche found in~\cite{AD}.

\begin{remark}\label{rmk:filling equiv}
To see how that our Proposition~\ref{prp:filling->pi} recovers the pure-infiniteness
statement in \cite[Proposition~3.9]{Suzuki}, we show that our notion of $n$-filling is
equivalent to Suzuki's notion of $n$-filling. Suzuki defines $G$ to be $n$-filling if for
every nonempty open $W$ there are open bisections $E_1, \dots, E_n$ such that $\bigcup_i
r(E_i W) = \Gunit$ (see \cite[Definition~3.2]{Suzuki}). If $G$ is $n$-filling in our
sense, then it is also $n$-filling in Suzuki's sense: given $W$, apply our definition
with $U_1 = U_2 = \dots = U_n = W$. The argument for the converse is essentially that of
\cite{JR}. If $G$ is $n$-filling in Suzuki's sense, then Proposition~\ref{min2} (or
\cite[Proposition~3.9]{Suzuki} and \cite[Proposition~5.7]{BCFS}) shows that $G$ is
minimal. So we can find $\gamma_{1,2} \in U_1 G U_2$ and then a compact open bisection
$F_{1,2} \subseteq U_1 G U_2$ containing $\gamma_{1,2}$. We then repeat this argument
with $U_1$ replaced by $r(F_{1,2})$ and $U_2$ replaced by $U_3$ to find a compact open
bisection $F_{1,3} \in r(F_{1,2}) G U_3$. Iteratively, we obtain compact open bisections
$F_{1,i} \subseteq \bigcap_{j < i} r(F_{1,j}) G U_i$. Apply Suzuki's $n$-filling property
with $W = \bigcap^n_{i=1} r(F_{1,i})$ to obtain compact open bisections $F'_i$ with
$\bigcup_i r(F'_i W) = \Gunit$. Now putting $E_i = F'_iF_{1,i}$, we obtain $\bigcup_i
r(E_iU_i) = \bigcup_i r(F'_i r(F_{1,i})) \supseteq \bigcup_i r(F'_i W) = \Gunit$.
\end{remark}

\subsection{Zero-dimensional topological graphs}

Topological graphs were defined and studied by Katsura in~\cite{Kat}. In this section we confine our focus to the zero-dimensional case and show that for a totally disconnected graph $E$ there is a natural semigroup $S(E)$ associated to $E$ which agrees with the groupoid type semigroup $S(G_E)$, where $G_E $ is the infinite-path groupoid as defined by Yeend \cite{Yeend} (though we will use the more familiar equivalent description given in \cite{RSWY}).  We use this to relate Theorem~\ref{theoremsfcnp} to Brown's theorem for crossed products of AF algebras.

Recall that a topological graph is a quadruple $E = (E^0, E^1, r, s)$ where $E^0$ and $E^1$ are locally compact Hausdorff spaces, $r : E^1 \to E^0$ (the range map) is continuous, and $s : E^1\to E^0$ (the source map) is a local homeomorphism.  For any $n=1,\dots,\infty$ we have the path spaces
\[E^n=\big\{(\lambda_{k})_{k=1}^{n}\ |\ \lambda_k\in E^1, s(\lambda_k)=r(\lambda_{k+1})\big\}\subseteq\prod_{k=1}^{n}E^1,\qquad E^*=\bigsqcup_{n\geq 0}E^n\]
endowed with the obvious topologies.  For any $n\in\mathbb{N}$ the range and source maps can be naturally extended (with the same properties) to $r, s: E^n\to E^0$. Moreover, the range map extends to infinite paths $r: E^\infty\to E^0$ via $r(\lambda_k)_{k=1}^{\infty}=r(\lambda_1)$. A subset $U \subseteq E^n$ for which $s|_U:U\to s(U)$ is a homeomorphism is called an $s$-section. Note that for any $n\in\mathbb{N}$ and  vertex $v\in E^0$, $s^{-1}(v)\subseteq E^{n}$ is discrete.

A topological graph $E$ is said to be totally disconnected if $E^0$ is totally disconnected. For the remainder of this section we restrict our attention to totally disconnected graphs whose range map is proper and surjective. In this setting the infinite path space $E^\infty$ is also totally disconnected. Here is a brief justification. For $n\in\mathbb{N}$ and $U \subseteq E^n$, we have the cylinder sets
\[Z(U) := \big\{\lambda\in E^{\infty}\ |\ \lambda=\alpha\beta, \alpha\in U, \beta\in E^{\infty}, s(\alpha)=r(\beta)  \big\} \subseteq E^\infty.\]
If $U\subseteq E^n$ is compact and open, one verifies that $Z(U)$ is compact open too, and a routine argument shows that the collection
\[\big\{Z(U)\ |\ U\subseteq E^n\ \text{ is a compact open $s$-section for some}\  n\geq 0 \big\}\] forms a basis for the topology on $E^\infty$. One can also show that every compact open $A\subseteq E^\infty$ can be written as a finite disjoint union $A=\sqcup_j Z(A_j)$ where the $A_j\subseteq E^{p_i}$ are compact open $s$-sections. These facts will prove useful in our work below.

The topological-graph \cstar-algebra $\cstar(E)$ is defined to be the Cuntz-Pimsner
algebra of a \cstar-correspondence constructed from the topological graph $E$. However,
it can also be realized as the reduced \cstar-algebra of a Deaconu-Reneault infinite-path
groupoid $G_E$. We briefly recall the construction of  $G_E$. Recall from Sections
2.3~and~2.4 of \cite{RSWY} that a boundary path of $E$ is either an infinite path or else
a finite path $\lambda\in E^*$ such that $s(\lambda) E^1$ is empty or has no compact
neighborhood in $E^1$. Since $r : E^1 \to E^0$ is proper and surjective by hypothesis, we
deduce that the boundary-path space of $E$ is precisely $E^\infty$.  The groupoid $G_E$
is defined as follows. The underlying set is
\[G_E:=\bigg\{\big(\alpha\lambda, |\alpha|-|\beta|,\beta\lambda\big)\ \big|\  \alpha,\beta \in
E^*, \lambda \in E^\infty, s(\alpha) = s(\beta) = r(\lambda)\bigg\}\subseteq E^\infty\times\mathbb{Z}\times E^\infty,\]
endowed with the subspace topology. The unit space is $G_E^{(0)}=\{(\lambda, 0, \lambda) : \lambda \in
E^\infty\}$ identified with $E^\infty$. The range and source maps are defined via
\[r,s :G_E\rightarrow G_E^{(0)}\qquad r(\mu,n,\nu) = \mu,\quad s(\mu,n,\nu) = \nu,\]
while the law of composition and inverse operation are given by
\[(\mu,m,\nu)(\nu,n,\lambda) =(\mu,m+n,\lambda),\quad\text{and}\quad (\mu,n,\nu)^{-1} = (\nu, -n, \mu).\]
The topology on $G_E$ has basic compact open bisections
\[Z(U, V) := \bigg\{(\alpha\lambda, |\alpha| - |\beta|, \beta\lambda)\ | \ \alpha \in U, \beta \in
V, \lambda \in E^\infty, r(\lambda) = s(\alpha) = s(\beta)\bigg\}.\]
where $U\subseteq E^n, V\subseteq E^m$ are compact open $s$-sections for some $m,n\in\mathbb{N}$. Note that in this context $r(Z(U,V))=Z(U)$ and $s(Z(U,V))=Z(V)$. Yeend  proves in~\cite{Yeend} that $\cstar(E) \cong \cstar(G_E)$.

We now define a semigroup $S(E)$ associated to a totally disconnected topological graph
$E$.  Again, we are restricting our attention to totally disconnected topological graphs
whose whose range map is proper and surjective. If $K\subseteq E^0$ is compact and $v\in
E^0$, then $r^{-1}(K)\cap s^{-1}(v)$ is the intersection of a compact set and a discrete
set, and is therefore finite. So for each integer $n \ge 0$ we define the function
$\Theta^n(f): E^0\to\mathbb{Z}$ by
\[\Theta^n(f)(v) = \sum_{\lambda \in E^n v} f(r(\lambda)).\]
Note that the sum runs over all $\lambda\in r^{-1}(\supp(f))\cap s^{-1}(v)$ which is a
finite set. The support of $\Theta^n(f)$ is compactly supported since
$\supp(\Theta^n(f))\subseteq s(r^{-1}(\supp(f)))$ which is compact. We therefore have an
operator $\Theta^n: C_c(E^0,\mathbb{Z})\to C_c(E^0,\mathbb{Z})$ which is a positive group
homomorphism and satisfies $\Theta^{n+m}=\Theta^{n}\circ\Theta^{m}$. Observe that if $E$
is a discrete directed graph, then $\Theta^n$ is just multiplication by the transpose
$(A_E^n)^t$ of the $n$th power of the adjacency matrix of $E$. It is useful to see how
$\Theta^n$ operates on characteristic functions. To that end, suppose $U\subseteq E^0$ is
a compact open subset. Then there are finitely many mutually disjoint compact open
$s$-sections $U_j$ such that
\[
UE^n:=\{\lambda\in E^n\ |\ r(\lambda)\in U\}=\bigsqcup_{j}^{k}U_j.
\]
So $\Theta^n(\ch_U)=\sum_{j=1}^{k}\ch_{s(U_j)}$.

\begin{definition}\label{defn:E equiv}
Let $E$ be a totally disconnected topological graph whose range map is proper and surjective.
\begin{enumerate}
\item [(i)] We define a relation $\sim$ on $C_c(E^0, \mathbb{Z})^+$ as follows: $f \sim g$ if there exist $p,q
\in \mathbb{N}$ such that $\Theta^p(f) = \Theta^q(g)$.
\item [(ii)] Define the relation $\approx$ on $C_c(E^0, \mathbb{Z})^+$ by $f \approx g$ if there exist finitely many pairs
$(f_i, g_i)_{i=1}^{n}$ in $C_c(E^0, \mathbb{Z})^+ \times C_c(E^0, \mathbb{Z})^+$ satisfying
\[f \sim\sum_{i=1}^n f_i,\quad g \sim \sum_{i=1}^n g_i\quad\text{and}\quad  f_i \sim g_i\quad\text{for each $i$}.\]
\end{enumerate}
\end{definition}

\begin{lemma}\label{lem:E equiv}
Let $E$ be a totally disconnected topological graph whose range map is proper and
surjective. The relations $\sim$ and $\approx$ are equivalence relations. If $f \approx
g$ then there exist finitely many compact open sets $U_i \subseteq E^0$, and natural
numbers $p, q$ and $q_i$ such that $\Theta^p(f) = \sum_i \Theta^{q_i}(\ch_{U_i})$ and
$\Theta^q(g) = \sum_i \ch_{U_i}$.
\end{lemma}
\begin{proof}
It is clear that both $\sim$ and $\approx$ are symmetric and reflexive. To see that $\sim$ is transitive, if $f\sim g \sim h$, say $\Theta^m(f) = \Theta^n(g)$ and $\Theta^p(g) = \Theta^q(h)$, then
$\Theta^{m+p}(f) = \Theta^{n+p}(g) = \Theta^{n+q}(h)$, giving $f \sim h$.

To see that $\approx$ is transitive, we first prove the final statement of the lemma.
If $f \approx g$ there are finitely many pairs $(f_i, g_i)_{i=1}\in C_c(E^0, \mathbb{Z})^+ \times C_c(E^0, \mathbb{Z})^+$, and natural numbers $a,b,c,d,p_i,r_i$ in $\mathbb{N}$ satisfying
\[
\Theta^a(f)=\Theta^b\big(\sum_i f_i\big),\quad \Theta^c(g)=\Theta^d\big(\sum_i g_i\big),\quad\text{and}\quad \Theta^{p_i}(f_i)=\Theta^{r_i}(g_i)\quad\text{for each}\  i.\]
Set $P := \max_i p_i$ and $p := a + d + P$. Then
\[
\Theta^{p}(f)
    = \Theta^{P+d}\bigg(\Theta^b\big(\sum_i f_i\big)\bigg)
    = \sum_i \Theta^{P - p_i + b + d}(\Theta^{p_i}(f_i))
    = \sum_i \Theta^{P - p_i + r_i + b + d}(g_i).
\]
So putting $q_i := P - p_i + r_i$, $q = b+c$ and $h_i = \Theta^{b+d}(g_i)$ for each $i$, we have
\[
\Theta^p(f) = \sum_i \Theta^{q_i}(h_i)\quad\text{and}\quad\Theta^q(g) =\Theta^{b}(\Theta^c(f)) = \Theta^b\bigg(\Theta^d\big(\sum_i g_i\big)\bigg) = \sum_i h_i.
\]
Now writing each $h_i = \sum_{j=1}^{m_i} \ch_{U_j}$, we obtain
\[
\Theta^p(f) = \sum_i \sum_{j \le m_i} \Theta^{q_i}(\ch_{U_i})
    \quad\text{and}\quad
\Theta^q(g) = \sum_i \sum_{j \le m_i} \ch_{U_i},
\]
so reindexing gives the final statement of the lemma.

Now suppose that $f \approx g \approx h$. By the preceding paragraph, we can choose
finitely many characteristic functions $g_1, \dots, g_M$ and $h_1, \dots, h_N$ of compact
open subsets of $E^0$, and natural numbers $p, q, a, b, q_i$ and $a_j$ such that
$\Theta^p(f) = \sum_i \Theta^{q_i}(g_i)$, $\Theta^q(g) = \sum_i g_i$, $\Theta^b(h) =
\sum_j \Theta^{a_j}(h_j)$ and $\Theta^a(g) = \sum_j h_j$.

We therefore have $\sum_i \Theta^a(g_i) = \Theta^{a+q}(g) = \sum_j \Theta^q(h_j)$. So
replacing each $g_i$ by $\Theta^a(g_i)$, $p$ by $p + a$, $q$ by $q + a$, each $h_j$ by
$\Theta^q(h_j)$ and $b$ by $b+q$, we obtain
\[
\Theta^p(f) = \sum_i \Theta^{q_i}(g_i),\quad
    \Theta^b(h) = \sum_j \Theta^{a_j}(h_j),\quad\text{ and }\quad
    \sum_i g_i = \sum_j h_j.
\]
Applying the refinement property to the equality $\sum_i g_i=\sum_j h_j$ we get
characteristic functions $\{k_{i,j}\}_{i,j}$ satisfying
\[g_i=\sum_j k_{i,j},\quad h_j=\sum_i k_{i,j}\]
for every $i$ and $j$. Consequently
\[\Theta^p(f) = \sum_i \Theta^{q_i}(g_i)= \sum_{i,j} \Theta^{q_i}(k_{i,j}),\qquad \Theta^b(h) =\sum_j \Theta^{a_j}(h_j)=\sum_{i,j} \Theta^{a_j}(k_{i,j}).\]
Now for every pair $(i,j)$ we have $\Theta^{q_i}(k_{i,j})\sim\Theta^{a_j}(k_{i,j})$, so by definition $f\approx h$.
%For $x \in E^0$, let $I_x = \{i : g_i(x) = 1\}$ and $J_x = \{j : h_j(x) = 1\}$. For each
%$I \subseteq \{1, \dots, M\}$ and each $J \subseteq \{1, \dots, N\}$, the set $W_{I,J} :=
%\{x : I_x = I\text{ and }J_x = J\}$ is compact and open because the $g_i$ and $h_i$ are
%locally constant. Since the $g_i$ an $h_j$ are characteristic functions, each $|I_x| =
%\sum_i g_i(x) = \sum_j h_j(x) = |J_x|$. Hence $|I| = |J|$ whenever $W_{I,J}$ is nonempty.
%We can therefore fix bijections
%\[\{\tau_{I,J} : I
%\to J \mid W_{I,J} \not= \emptyset\}.\]
%Now for $I,J$ with $W_{I,J}$ nonempty and for $i
%\in I$, we have $g_i|_{W_{I,J}} = \ch_{W_{I,J}} = h_{\tau_{I, J}(i)}|_{W_{I, J}}$. Thus
%\[
%\sum_i g_i
 %   = \sum_{i, I, J} g_i|_{W_{I, J}}
%    = \sum_{i, I, J} h_{\tau_{I,J}(i)}|_{W_{I, J}}
%    = \sum_j h_j.
%\]
%Consequently
%\[
%f \sim \sum_i \Theta^{q_i}(g_i)
%    = \sum_{i, I, J} \Theta^{q_i}(g_i|_{W_{I, J}}),
%\]
%and likewise $h \sim \sum_{i, I, J} \Theta^{a_i}(h_{\tau_{I, J}(i)}|_{W_{I, J}})$. Since each $g_i|_{W_{I, J}} = h_{\tau_{I, J}(i)}|_{W_{I, J}}$, we deduce that each
%$g_i|_{W_{I,J}} \sim h_{\tau_{I, J}(i)}|_{W_{I, J}}$ and may conclude that $f \approx h$.
\end{proof}

\begin{definition}
Let $E$ be a totally disconnected topological graph whose range map is proper and surjective. We define the \emph{type semigroup of $E$} as \begin{equation}
S(E):=C_c(E^0,\mathbb{Z})^{+}/{\approx}
\end{equation}\label{eq:top graph S}
and write $[f]_{E}$ for the equivalence class with representative $f\in C_c(E^0,\mathbb{Z})^{+}$.
\end{definition}

Since the relation $\approx$ clearly respects addition in $C_c(E^0, \mathbb{Z})^+$, we see that $S(E)$ is indeed an abelian monoid under the operation $[f]_E + [g]_E = [f+g]_E$. As usual we give $S(E)$ the algebraic ordering.

Our next goal is to prove that this $S(E)$ coincides with the semigroup $S(G_{E})$ constructed from the ample groupoid $G_E$ associated to $E$.

\begin{proposition}\label{prp:SE=SGE}
Let $E$ be a totally disconnected topological graph whose range map is proper and
surjective. There is an isomorphism of monoids $\tau : S(E) \to S(G_E)$ such that
$\tau([\ch_U]_E) = [\ch_{Z(U)}]_{G_{E}}$ for every compact open $U \subseteq E^0$.
\end{proposition}

\begin{proof}
Since $r: E^\infty\to E^0$ is proper and surjective, the map $\tilde\tau : C_c(E^0, \mathbb{Z})^+ \to C_c(G_E^{(0)}, \mathbb{Z})^+$ defined by $\tilde\tau(f)=f\circ r$ is a well-defined, injective monoid homomorphism. One immediately observes that $\tilde\tau(\ch_U)=\ch_{U}\circ r=\ch_{Z(U)}$ for every compact open $U\subseteq E^0$.

To see that $\tilde\tau$ descends to $\tau : S(E) \to S(G_E)$, we must show that if $f
\approx g$, then $\tilde\tau(f) \sim_{G_E} \tilde\tau(g)$. We first show that $\tilde\tau(f)
\sim_{G_E} \tilde\tau(\Theta^n(f))$ for any $f \in C_c(E^0, \mathbb{Z})^+$ and $n \in
\mathbb{N}$. Since $\Theta^n$ and $\tilde\tau$ are additive, and $\sim_{G_E}$ respects addition, it suffices to consider $f=\ch_{U}$ for a compact open subset $U\subseteq E^0$. We write $UE^n=\bigsqcup_{j}^{k}U_j$ where the $U_j\subseteq E^n$ are compact open $s$-sections. Also set as bisections in $G_E$, $F_j:=Z(U_j, s(U_j))$. We then have
\begin{align*}
\tilde\tau(\Theta^n(\ch_U))&=\tilde\tau\bigg(\sum_{j=1}^{k}\ch_{s(U_j)}\bigg)=\sum_{j=1}^{k}\tilde\tau(\ch_{s(U_j)})=\sum_{j=1}^{k}\ch_{Z(s(U_j))}\\&=\sum_{j=1}^{k}\ch_{s(F_j)}\sim_{G_E}\sum_{j=1}^{k}\ch_{r(F_j)}=\sum_{j=1}^{k}\ch_{Z(U_j)}=\ch_{\sqcup_{j}^{k}Z(U_j)}=\ch_{Z(U)}=\tilde\tau(\ch_{U})
\end{align*}
%Indeed, given $f$ and $p$, for each of the finitely many values $n$ in the
%image of $f$, partition $f^{-1}(n)E^p$ into finitely many pairwise disjoint compact open
%$s$-sections $V_{n,1}, \dots, V_{n, k_n}$. Let $V_1, \dots, V_l$ be a listing of the sets
%$\{V_{n,i} : n \in \operatorname{image}(f), i \le k_n\}$ with each set $V_{n,i}$ listed
%$n$ times. Then $f = \sum_{j=1}^l \ch_{r(V_j)}$, and $\Theta^p(f) = \sum_j \ch_{s(V_j)}$.
%Hence $\tilde\tau(f) = \sum^l_{j=1} \ch_{r(Z(V_j, s(V_j)))}$ and $\tilde\tau(\Theta^p(f))
%= \sum^l_{j=1} \ch_{s(Z(V_j, s(V_j)))}$, giving $\tilde\tau(f) \sim
%\tilde\tau(\Theta^p(f))$.
Since $\sim_{G_E}$ is an equivalence relation on $C_c(G_E^{(0)}, \mathbb{Z})^+$, we deduce that
if $f \sim g$ in $C_c(E^0, \mathbb{Z})^+$ then $\tilde\tau(f) \sim_{G_E} \tilde\tau(g)$. Now
suppose that $f \approx g \in C_c(E^0, \mathbb{Z})^+$, then there exist finitely many pairs
$(f_i, g_i)_{i=1}^{n}$ in $C_c(E^0, \mathbb{Z})^+ \times C_c(E^0, \mathbb{Z})^+$ satisfying
\[f \sim\sum_i^n f_i,\quad g \sim \sum_i^n g_i\quad\text{and}\quad  f_i \sim g_i\quad\text{for each $i$}.\]
Then
\[\tilde\tau(f)\sim_{G_E}\tilde\tau\bigg(\sum_i f_i\bigg)=\sum_i \tilde\tau(f_i)\sim_{G_E}\sum_i \tilde\tau(g_i)=\tilde\tau\bigg(\sum_i g_i\bigg)\sim_{G_E}\tilde\tau(g).\]
%The final statement of
%Lemma~\ref{lem:E equiv} gives finitely many $h_i \in C_c(E^0, \mathbb{Z})_+$, and natural
%numbers $p, q$ and $q_i$ such that $\Theta^p(f) = \sum_i \Theta^{q_i}(h_i)$ and such that
%$\Theta^q(g) = \sum_i h_i$. By the preceding paragraph, we have $\tilde\tau(f) \sim
%\sum_i \tilde\tau(\Theta^{q_i}(h_i))$ and $\tilde\tau(g) \sim \sum_i h_i$. So since
%$\sim$ is an equivalence relation, we just have to show that $\sum_i
%\tilde\tau(\Theta^{q_i}(h_i)) \sim \tilde\tau(g) \sim \sum_i h_i$. The preceding
%paragraph also shows that each $\tilde\tau(\Theta^{q_i}(h_i)) \sim \tilde\tau(h_i)$, so
%we can choose compact open bisections $U_{i,1}, \dots U_{i,k_i}$ such that
%$\tilde\tau(\Theta^{q_i}(h_i)) = \sum_{j} \ch_{r(U_{i,j})}$ and $\tilde\tau(h_i) = \sum_j
%\ch_{s(U_{i,j})}$. Hence
%\[
%\sum_i \tilde\tau(\Theta^{q_i}(h_i))
 %   = \sum_{i,j} \ch_{r(U_{i,j})}
 %   \equiv \sum_{i,j} \ch_{s(U_{i,j})}
  %  = \sum_i \tilde\tau(h_i).
%\]
So $\tilde\tau$ induces a homomorphism $\tau:S(E) \to S(G_E)$ satisfying $\tau([\ch_U]_E)
= [\ch_{Z(U)}]_{G_E}$.

To show surjectivity, it suffices to verify that $[\ch_{A}]_{G_E}$ is in the image of $\tau$ for any compact open $A\subseteq E^\infty$. Since any such $A$ can be written as a finite disjoint union $A=\sqcup_j Z(B_j)$ where the $B_j\subseteq E^{p_i}$ are compact open $s$-sections, we need only show that such $[Z(B)]_{G_E}\in\im(\tau)$ for such an $s$-section $B$. To that end
\[ [\ch_{Z(B)}]_{G_E} = [\ch_{r(Z(B, s(B)))}]_{G_E} = [\ch_{s(Z(B, s(B)))}]_{G_E} = [\ch_{Z(s(B))}]_{G_E} = \tau([\ch_{s(B)}]_E),\]
so $\tau$ is indeed surjective.

Lastly we show that $\tau$ is injective. Suppose $f,g\in C_c(E^0,\mathbb{Z})^+$ and that $\tilde\tau(f) \sim_{G_E} \tilde\tau(g)$. We need to show that $f\approx g$. Choose compact open bisections
$U_i\subset G_E$ such that
\[\tilde\tau(f) = \sum_i \ch_{r(U_i)},\quad\text{and}\quad \tilde\tau(g) = \sum_i\ch_{s(U_i)}.\]
By definition of the topology on $G_E$, we can decompose each $U_i$ as
$U_i = \bigsqcup_{j=1}^{n_i} Z(V_j, W_j)$ where $V_j\subset E^{p_j}$, $W_j\subseteq E^{q_j}$ are compact open $s$-sections for some $p_j, q_j\in\mathbb{N}$, and $s(V_j) =s(W_j)$. By relabeling, we can assume that each $U_i$ is equal to such a $Z(V_i, W_i)$. By taking $p = \max_i p_i$, covering each $s(V_i)E^{p - p_i}$ with mutually
disjoint compact open $s$-sections $\{Y_{i,j} : j \le m_i\}$ and then replacing each
$Z(V_i, W_i)$ with $\{Z(V_i Y_{i,j}, W_i Y_{i,j}) : j \le m_i\}$, we can moreover assume that the $p_i$ are all equal.

Since $\tilde\tau(f)(\lambda x) = f(r(\lambda)) = \tilde\tau(f)(\lambda' y)$ whenever
$\lambda,\lambda' \in E^p$ satisfy $r(\lambda) = r(\lambda')$, we see that for each
$\lambda \in \operatorname{supp}(f)E^p$, we have $|\{i : \lambda \in V_i\}| = f(x)$. It
follows that $\tilde\tau(\ch_{s(V_i)}) = \sum_i \ch_{Z(s(V_i))} =
\tilde\tau(\Theta^p(f))$. We have $s(V_i) = s(W_i)$ for all $i$, and so $\sum_i
\ch_{Z(s(W_i))} = \tilde\tau(\Theta^p(f))$ as well. Each $\ch_{Z(s(W_i))} =
\tilde\tau(\ch_{s(W_i)})$, so it suffices to show that $g \approx \sum_i \ch_{s(W_i)}$.
So, putting $q = \max_i q_i$, it suffices to show that $g \approx \sum_i \Theta^{q -
q_i}(\ch_{s(W_i)})$.

By covering each $s(W_i)E^{q - q_i}$ by mutually disjoint compact open $s$-sections
$Y_{i,j}$, we can write $\sum_i \tilde\tau(\Theta^{q - q_i}(\ch_{s(W_i)})) = \sum_{i,j}
\ch_{s(W_i Y_{i,j})}$. Now the bisections $B_{i,j} = Z(W_i Y_{i,j}, s(Y_{i,j}))$ satisfy
$\tilde\tau\Big(\sum_i \Theta^{q - q_i}(\ch_{s(W_i)})\Big) = \sum_{i,j} \ch_{s(B_{i,j})}$
and $\tilde\tau(g) = \sum_{i,j} \ch_{r(B_{i,j})}$. Since each $W_i Y_{i,j} \subseteq
E^q$, the argument of the preceding paragraph shows that
\[
\tilde\tau\Big(\sum_i \Theta^{q - q_i}(\ch_{s(W_i)})\Big)
    = \sum_{i,j} \ch_{Z(s(B_{i,j}))}
    = \tilde\tau(\Theta^q(g)).
\]
Since $\tilde\tau$ is injective, it follows that $\Theta^q(g) = \sum_i \Theta^{q -
q_i}(\ch_{s(W_i)})$, and so $g \approx \sum_i \Theta^{q - q_i}(\ch_{s(W_i)})$ as
required.
\end{proof}

We conclude with a dichotomy result for totally disconnected topological graphs.

\begin{corollary}
Let $E$ be a totally disconnected topological graph whose range map is proper and surjective. If $\cstar(E)$ is simple and the semigroup $S(E)$ of~\eqref{eq:top graph S} is almost unperforated, then $\cstar(E)$ is either purely infinite or quasidiagonal.
\end{corollary}
\begin{proof}
Since $\cstar(G_E) \cong \cstar(E)$ is simple, \cite[Theorem~5.1]{BCFS} shows that $G_E$ is minimal and topologically principal. Proposition~\ref{prp:SE=SGE} shows that $S(G_E)$ is almost unperforated. The result now follows from Theorem~\ref{dichotomy}.
\end{proof}

\end{document}